\documentclass[a4paper]{article}

\usepackage[latin1]{inputenc} 
\usepackage{graphicx}
\usepackage[pdftex,bookmarks]{hyperref}
\usepackage{amsmath}    
\usepackage{amssymb}
\usepackage{amsmath,amsthm}
\usepackage{textcomp}
\usepackage[all]{xy}

\newtheorem{teo}{Theorem}
\newtheorem{defi}{Definition}

\newtheorem{cor}{Corollary}
\newtheorem{prop}{Proposition}
\newtheorem{exa} {Example}
\newtheorem{rem} {Remark}

\title{\huge \bf General perversities and $L^2$ de Rham and Hodge theorems for stratified pseudomanifolds}
\author{Francesco Bei}
\date{}

\begin{document}

\maketitle
 \pagestyle{myheadings} \markboth {\centerline{\small \tt Francesco Bei}}{\centerline{\small \tt General perversities and $L^2$ de Rham and Hodge theorems for stratified pseudomanifolds}}

\bigskip

\begin{abstract}
Given  a compact stratified pseudomanifold $X$ with a Thom-Mather stratification and a class of riemannian metrics over its regular part, we study the relationships  between the $L^{2}$ de Rham and Hodge cohomology of the regular part of $X$  and the intersection cohomology of $X$ associated to some perversities.  More precisely, to a kind of metric which we call  \textbf{quasi edge with weights}, we associate two \textbf{general perversities} in the sense of G. Friedman, $p_{g}$ and its dual $q_{g}$. We then  show that:
\begin{enumerate}
\item The absolute $L^{2}$ Hodge cohomology is isomorphic to the maximal $L^{2}$ de Rham cohomology and  this is in turn isomorphic to the intersection cohomology associated to the perversity $q_{g}$. 
\item The relative $L^{2}$ Hodge cohomology is isomorphic to the minimal $L^{2}$ de Rham cohomology and  this is in turn isomorphic to the intersection cohomology associated to the perversity $p_{g}$.
\end{enumerate}
Moreover we give a partial answer to the inverse question: given $p$, a general perversity in the sense of Friedman on $X$, is there a riemannian metric $g$ on $reg(X)$ such that a $L^{2}$ de Rham and Hodge theorem hold for $g$ and $p$? We then show that the answer is positive in the following two cases: if $p$ is greater or equal to the upper middle perversity or if it is smaller or equal to the lower middle one. Finally we conclude giving several corollaries about the properties of these $L^{2}$ Hodge and de Rham cohomology groups.
\end{abstract}


\section*{Introduction}
Let $X$ be a compact stratified pseudomanifold and let $reg(X)$ be its regular part.  The study of the relationships between the $L^{2}$ de Rham and Hodge cohomology associated to a given riemannian metric and the intersection cohomology of $X$ has a long history initiated at the end of seventies with  the celebrated papers of J. Cheeger \cite{JC} and \cite{C}. In \cite{JC} Cheeger established a Hodge theorem for manifolds with isolated conic singularities; in \cite{C} he showed that if $X$ is a closed Witt $PL$ stratified pseudomanifold and if $g$ is an admissible riemannian metric on $reg(X)$ 
 then the $L^{2}$ maximal Hodge cohomology is finite dimensional and isomorphic to the maximal $L^{2}$ de Rham cohomology.  Furthermore, without the Witt assumption but using some additional hypothesis about the calculation of the maximal $L^{2}$ cohomology of a cone over a riemannian manifold, see lemma 3.4 in \cite{C}, he showed that  the maximal $L^{2}$ de Rham cohomology is  isomorphic to the intersection cohomology of $X$ associated to the lower middle perversities. Subsequently this $L^{2}$ de Rham theorem of Cheeger was generalized by M. Nagase, which in \cite{MN} showed that given a perversity $p\leq \underline{m}$, where $\underline{m}$ is the lower middle perversity, it is  possible to construct over the regular part of $X$ a riemannian metric  $g$ associated to the the perversity $p$ such that the maximal $L^{2}$ de Rham cohomology is isomorphic to the intersection cohomology of $X$ associated to the perversity $p$. In both these papers the proofs of the $L^{2}$ de Rham theorem were done by constructing a subcomplex of the complex of $L^{2}$ differential form with weak differential quasi-isomorphic to it and integrating the forms of this subcomplex over some $PL-$chains. Afterwards, in the paper \cite{MAN}, Nagase presented a new proof of his $L^{2}$ de Rham Theorem that employed the sheaf-theoretic approach of Goresky-MacPherson \cite{GMA}\footnote{In \cite{C}, \cite{MN} and \cite{MAN} the symbol $\overline{m}$ is used for the lower middle perversity.}. 
\\Recently R. Mazzeo and E. Hunsicker proved \cite{HM} a $L^{2}$ de Rham and Hodge theorem  on a manifold with edges. We recall that a manifolds with edges  is a compact stratified pseudomanifold of depth one, $X\supset B,\ B=\bigcup_{j} B_{j} $. For each stratum $B_{j}$, which in this case is just a closed manifold, there exists an open neighbourhood $U_{j}$ of $B_{j}$ in $X$ which is diffeomorphic to a bundle of cones, that is, a bundle with basis $B_{j}$   and fibers $C(F_{j})$ with $F_{j}$ a closed manifold that depends only on $B_{j}$.  Over $X-B$ they consider an edge metric $g$, that is, a riemannian metric   such that over each $U_{j}-B_{j}$ it is quasi-isometric to $dr\otimes dr+\pi_{j}^{*}h_{j}+r^{2}k_{j}$ where $k_{j}$ is a two symmetric tensor field which restricts to a metric on each fiber $F_{j}$, $\pi_{j}:U_{j}\rightarrow B_{j}$ is the projection and $h_{j}$ is a riemannian metric on $B_{j}$. Then for the maximal and minimal $L^{2}$ de Rham cohomology and for the absolute and relative Hodge cohomology the following isomorphisms holds :
$$ I^{\underline{m}}H^{i}(X)\cong H^{i}_{2,max}(reg(X),g)\cong \mathcal{H}^{i}_{abs}(reg(X),g)$$
$$ I^{\overline{m}}H^{i}(X)\cong H^{i}_{2,min}(reg(X),g)\cong\mathcal{H}^{i}_{rel}(reg(X),g)
\footnote{In \cite{HM} the first isomorphism involves the upper middle perversity and  the second involves the lower middle perversity. The reason is that in \cite{HM} the definitions of these perversities are reversed from the usual ones. }$$This result was later generalized by Hunsicker \cite{H}. Given a manifold with edges with only one singular stratum $B$, Hunsicker considers a riemannian metric $g$ on $reg(X)$ such that over  $U-B$ it is quasi-isomorphic to $$dr\otimes dr+\pi^{*}h+r^{2c}k\ \text{where}\ 0<c\leq 1.$$ The isomorphisms between the $L^{2}$ de Rham, the Hodge and the intersection cohomology of $X$ that she gets, for this kind of metrics, are the following:
$$\mathcal{H}^{i}_{abs}(reg(X),g)\cong H^{i}_{2,max}(reg(X),g)\cong \left\{
\begin{array}{ll}
 I^{\underline{m}-[[\frac{1}{2c}]]}H^{i}(X)&  f\ is\ even\\
I^{\underline{m}-[[\frac{1}{2}+\frac{1}{2c}]]}H^{i}(X)&  f\ is\ odd
\end{array}
\right.
$$

$$\mathcal{H}^{i}_{rel}(reg(X),g)\cong H^{i}_{2,min}(reg(X),g)\cong \left\{
\begin{array}{ll}
 I^{\overline{m}+[[\frac{1}{2c}]]}H^{i}(X)&  f\ is\ even\\
I^{\overline{m}+[[\frac{1}{2}+\frac{1}{2c}]]}H^{i}(X)&  f\ is\ odd
\end{array}
\right.
$$
where $[[x]]$ denotes the greatest integer strictly less than $x$.\footnote{ Also in this case there is a switch of perversities from \cite{H}. It is caused by the fact that in \cite{H} the intersection cohomology for a perversity $p$ is the cohomology of the Deligne sheaf for such perversity or equivalently the cohomology of the complex of the intersection chain sheaves for the perversity $p$.} It is immediate to note that when $c=1$ then $[[\frac{1}{2}]]=[[\frac{1}{2}+\frac{1}{2c}]]=0$ and then this result reduces to the  results in \cite{HM}.\\ We note that all the previous results we recalled can be interpreted in two different way: on the one hand they assert that for certain riemannian metrics on $reg(X)$ the $L^{2}$ de Rham and Hodge cohomology groups associated to them are isomorphic to the intersection cohomology groups associated certain perversities; therefore these $L^{2}$ de Rham and Hodge cohomology groups  do not depend from the metrics chosen but only from the stratified homotopy class of $X$ and from the perversity associated to the metrics. On the other hand the previous results assert that for some perversities the intersection cohomology groups associated to them are constructible in a analytic way. In other words there is a riemannian metric on $reg(X)$ such that a $L^{2}$ de Rham and Hodge theorem holds for the perversity considered. \vspace{1 cm} 

The main goal of this paper is to investigate the two following questions:
\begin{enumerate} 
\item Is it possible to generalize the  result established by Hunsicker in the  edge case to the case of any compact and oriented smoothly stratified pseudomanifold with a Thom-Mather stratification?
\item Given $p$, a general perversity in the sense of Friedman on $X$, is there a riemannian metric $g$ on $reg(X)$ such that a $L^{2}$ de Rham and Hodge theorem holds for them?
\end{enumerate}

We give a positive answer to the first question and we show that if $p$ is greater or equal to the upper middle perversity or smaller or equal to the lower middle one then also the second question has a positive answer. In particular this last result generalizes the result of Nagase in \cite{MN}. \\
More precisely given $X$, a compact and oriented smoothly stratified pseudomanifold with a Thom-Mather stratification, we consider a riemannian metric $g$ over its regular part, $reg(X)$, that satisfies the following properties:

\begin{enumerate}
\item Take any stratum $Y$ of $X$;  for each $q\in Y$ there exist an open neighbourhood $U$ of $q$ in $Y$ such that $\phi:\pi_{Y}^{-1}(U)\rightarrow U\times C(L_{Y})$ is a stratified isomorphism; in particular $\phi:\pi_{Y}^{-1}(U)\cap reg(X)\rightarrow U\times reg(C(L_{Y}))$ is a diffeomorphism. Then, for each $q\in Y$, there exists one of these trivializations $(\phi,U)$ such that $g$ restricted on $\pi_{Y}^{-1}(U)\cap reg(X)$ satisfies the following properties:
$$(\phi^{-1})^{*}(g|_{\pi_{Y}^{-1}(U)\cap reg(X)})\cong dr\otimes dr+h_{U}+r^{2c}g_{L_{Y}}$$
 where $h_{U}$ is a riemannian metric defined over $U$, $c\in \mathbb{R}$ and $c>0$, $g_{L_{Y}}$ is a riemannian metric   on $reg(L_{Y})$, $dr\otimes dr+h_{U}+r^{2c_{Y}}g_{L_{Y}}$ is a riemannian metric of product type on $U\times reg(C(L_{Y}))$ and with $\cong$ we mean \textbf{quasi-isometric}. 
 \item If $p$ and $q$ lie in the same stratum $Y$ then in \eqref{yhnn} there is the same weight. We label it $c_{Y}$. 
\end{enumerate}
We call such kind of riemannian metric \textbf{ quasi edge metric with weights}.\\ To these we associate a \textbf{general perversity} $p_{g}$ in the sense of G. Friedman:
$$p_{g}(Y):= Y\longmapsto [[\frac{l_{Y}}{2}+\frac{1}{2c_{Y}}]]= \left\{
\begin{array}{lll}
0 & l_{Y}=0\\
\frac{l_{Y}}{2}+[[\frac{1}{2c_{Y}}]] & l_{Y}  even,\ l_{Y}\neq 0\\
\frac{l_{Y}-1}{2}+[[\frac{1}{2}+\frac{1}{2c_{Y}}]]& l_{Y}\ odd
\end{array}
\right.
$$
where $l_{Y}=dimL_{Y}$ and, given any real and positive number $x$, $[[x]]$ is the greatest integer strictly less than $x$.\\
The isomorphisms between the $L^{2}$ de Rham,  the Hodge and the intersection cohomology that we get  are then:
\begin{equation}
 I^{q_{g}}H^{i}(X,\mathcal{R}_{0})\cong H_{2,max}^{i}(reg(X),g)\cong \mathcal{H}_{abs}^{i}(reg(X),g)
\label{fff}
\end{equation}
\begin{equation}
I^{p_{g}}H^{i}(X, \mathcal{R}_{0})\cong H_{2,min}^{i}(reg(X),g)\cong \mathcal{H}_{rel}^{i}(reg(X),g)
\label{uuu}
\end{equation}
where $q_{g}$ is the complementary perversity of $p_{g}$, that is $q_{g}=t-p_{g}$ with $t$  the usual top perversity.  $\mathcal{R}_{0}$ is  the stratified coefficient system made  of  the pair of coefficient systems given by $(X-X_{n-1})\times \mathbb{R}$ over $X-X_{n-1}$ where the fibers $\mathbb{R}$ have the discrete topology  and the constant $0$ system on $X_{n-1}$. In particular, for all $i=0,...,n$ the groups $$H_{2,max}^{i}(reg(X),g),\ H_{2,min}^{i}(reg(X),g),\  \mathcal{H}_{abs}^{i}(reg(X)),\  \mathcal{H}_{rel}^{i}(reg(X))$$ are all finite dimensional. Note that in this paper we allow for  the existence of one codimensional strata; furthermore $p_{g}$ and $q_{g}$ are not  classical perversities in the sense of Goresky-MacPherson. This is why we have to replace the coefficient $\mathbb{R}$ with $\mathcal{R}_{0}$. It will be shown in corollary \ref{z} that if $p_{g}$ and $q_{g}$ are classical perversities in the sense of Goresky-MacPherson and $X_{n-1}=X_{n-2}$ then it is possible to replace $\mathcal{R}_{0}$ with $\mathbb{R}$. It is immediate to note that when $X$ is a manifold with edges with only one singular stratum this result reduces to the one proved by Hunsicker in \cite {H}.\\ Moreover we show that:
\begin{enumerate}
 \item if $p$ is a general perversity on $X$ in the sense of Friedman such that $p\geq \overline{m}$, where $\overline{m}$ is the upper middle perversity , and such that $p(Y)=0$ for each stratum with $cod(Y)=1$, then it is possible to construct on $reg(X)$ a quasi edge metric with weights $g$ such that \eqref{uuu} holds.
\item if $q$ is a general perversity on $X$ in the sense of Friedman such that $p\leq \underline{m}$, where $\underline{m}$ is the lower middle perversity , and such that $p(Y)=-1$ for each stratum with $cod(Y)=1$, then it is possible to construct on $reg(X)$ a quasi edge metric with weights $g$ such that \eqref{fff} holds.
 \end{enumerate}
 Finally we conclude the paper giving several corollaries about the properties of these $L^{2}$ de Rham and Hodge cohomology groups and about the properties of some  operators  associated to the metric $g$. We point out that these results can be used to study the perverse signatures of $X$, as it is shown in \cite{H} when the stratified pseudomanifold $X$ has only one singular stratum, see also \cite{CD}.

The paper is structured in the following way: in the first part we recall  notions which are fundamental to the whole work such as Hilbert complexes, intersection homology, intersection homology with general perversity, as defined by G. Friedman \cite{GP} and \cite{IGP} and stratified pseudomanifolds with a Thom-Mather stratification. We also introduce the riemannian metrics  which we will use for the rest of the paper and the general perversities associated to them. The second part contains some results  needed in order to calculate the maximal $L^{2}$ de Rham cohomology of a cone over a riemannian manifold endowed with a conic metric. The third part contains the calculation of the maximal $L^{2}$ de Rham cohomology of a cone over a riemannian manifold endowed with a conic metric with weights. Finally the last part contains the results that we have  announced above, their proofs and several corollaries. For the proof of the isomorphims \eqref{fff}, \eqref{uuu} in the last section we use a sheaf-theoretic point of view as is \cite{H}, \cite{HM} and \cite{MN}. More precisely to show the isomorphism \eqref{fff} we will construct a complex of fine sheaves  whose hypercohomology is the maximal $L^{2}$ de Rham cohomology and we will show that such complex satisfy the generalization given by Friedman of the theorem of Goresky and MacPherson in \cite{GMA}. Finally using some duality results we will get the isomorphisms \eqref{uuu}.  \vspace{1 cm}

\section{Background}

\subsection{Hilbert complexes}

In this first subsection we recall the notion of Hilbert complex  following \cite{HM}.

\begin{defi} A Hilbert complex is a complex, $(H_{*},D_{*})$ of the form:
\begin{equation}
0\rightarrow H_{0}\stackrel{D_{0}}{\rightarrow}H_{1}\stackrel{D_{1}}{\rightarrow}H_{2}\stackrel{D_{2}}{\rightarrow}...\stackrel{D_{n-1}}{\rightarrow}H_{n}\rightarrow 0,
\label{mm}
\end{equation}
where each $H_{i}$ is a separable Hilbert space and each map $D_{i}$ is a closed operator called the differential such that:
\begin{enumerate}
\item $\mathcal{D}(D_{i})$, the domain of $D_{i}$, is dense in $H_{i}$.
\item $ran(D_{i})\subset \mathcal{D}(D_{i+1})$.
\item $D_{i+1}\circ D_{i}=0$ for all $i$.
\end{enumerate}
\end{defi}
The cohomology groups of the complex are $H^{i}(H_{*},D_{*}):=Ker(D_{i})/ran(D_{i-1})$. If the groups $H^{i}(H_{*},D_{*})$ are all finite dimensional we say that it is a  $Fredholm\ complex$.

Given a Hilbert complex there is a dual Hilbert complex
\begin{equation}
0\leftarrow H_{0}\stackrel{D_{0}^{*}}{\leftarrow}H_{1}\stackrel{D_{1}^{*}}{\leftarrow}H_{2}\stackrel{D_{2}^{*}}{\leftarrow}...\stackrel{D_{n-1}^{*}}{\leftarrow}H_{n}\leftarrow 0,
\label{mmp}
\end{equation}
defined using $D_{i}^{*}:H_{i+1}\rightarrow H_{i}$, the Hilbert space adjoints of the differentials\\ $D_{i}:H_{i}\rightarrow H_{i+1}$. The cohomology groups of $(H_{j},(D_{j})^*)$, the dual Hilbert  complex, are $$H^{i}(H_{j},(D_{j})^*):=Ker(D_{n-i-1}^{*})/ran(D_{n-i}^*).$$\\For all $i$ there is also a laplacian $\Delta_{i}=D_{i}^{*}D_{i}+D_{i-1}D_{i-1}^{*}$ which is a self-adjoint operator on $H_{i}$ with domain $$\mathcal{D}(\Delta_{i})=\{v\in \mathcal{D}(D_{i})\cap \mathcal{D}(D_{i-1}^{*}): D_{i}v\in \mathcal{D}(D_{i}^{*}), D_{i-1}^{*}v\in \mathcal{D}(D_{i-1})\}$$ and nullspace: $$\mathcal{H}_{i}(H_{*},D_{*}):=ker(\Delta_{i})=Ker(D_{i})\cap Ker(D_{i-1}^{*}).$$

The following propositions are  standard results for these complexes. The first result is a weak Kodaira decomposition:

\begin{prop} [\cite{BL}, Lemma 2.1] Let $(H_{i},D_{i})$ be a Hilbert complex and $(H_{i},(D_{i})^{*})$ its dual complex, then: $$H_{i}=\mathcal{H}_{i}\oplus\overline{ran(D_{i-1})}\oplus\overline{ran(D_{i}^{*})}.$$
\label{kio}
\end{prop}



\begin{prop}[\cite{BL}, corollary 2.5] If the cohomology of a Hilbert complex $(H_{*}, D_{*})$ is finite dimensional then, for all $i$,  $ran(D_{i-1})$ is closed  and  $H^{i}(H_{*},D_{*})\cong \mathcal{H}^{i}(H_{*},D_{*}).$
\label{qwe}
\end{prop}

\begin{prop}[\cite{BL}, corollary 2.6] A Hilbert complex  $(H_{j},D_{j}),\ j=0,...,n$ is a Fredholm complex if and only if  its dual complex, $(H_{j},D_{j}^{*})$, is Fredholm. If this is the case then
\begin{equation}
\mathcal{H}_{i}(H_{j},D_{j})\cong H_{i}(H_{j},D_{j})\cong H_{n-i}(H_{j},(D_{j})^{*})\cong \mathcal{H}_{n-i}(H_{j},(D_{j})^{*})
\end{equation}
\label{fred}
\end{prop}


The final result that we recall shows that is possible to compute these cohomology groups using a core subcomplex $$\mathcal{D}^{\infty}(H_{i})\subset H_{i}.$$ For all $i$ $\mathcal{D}^{\infty}(H_{i})$ consists of all elements $\eta$ that are in the domain of $\Delta_{i}^{l}$ for all $l\geq 0.$

\begin{prop}[\cite{BL}, Theorem 2.12] The complex $(\mathcal{D}^{\infty}(H_{i}),D_{i})$ is a subcomplex quasi-isomorphic to the complex $(H_{i},D_{i})$
\label{dff}
\end{prop} 

The main case of interest here is when $(M,g)$ is a (not necessarily complete) riemannian manifolds, $H_{i}=L^{2}\Omega^{i}(M,g)$, and $D_{i}$ is the exterior derivative operator.\\ Consider the de Rham complex $(C_{0}^{\infty}\Omega^{*}(M),d_{*})$ where each form $\omega\in C_{0}^{\infty}\Omega^{i}(M)$ is a $i-$form with compact support. To turn this complex into a Hilbert complex we must specify a closed extension of $d$. With the two following propositions we will recall the two canonical closed extensions of $d$

\begin{defi} The maximal extension $d_{max}$; this is the operator acting on the domain:
\begin{equation} 
\mathcal{D}(d_{max,i})=\{\omega\in L^{2}\Omega^{i}(M,g): \exists\ \eta\in L^{2}\Omega^{i+1}(M,g)
\end{equation}

$$ s.t.\ <\omega,\delta_{i}\zeta>_{L^{2}(M,g)}=<\eta,\zeta>_{L^{2}(M,g)}\ \forall\ \zeta\in C_{0}^{\infty}\Omega^{i+1}(M) \}$$

In this case $d_{max,i}\omega=\eta.$ In  other words $\mathcal{D}(d_{max,i})$ is the largest set of forms $\omega\in L^{2}\Omega^{i}(M,g)$ such that $d_{i}\omega$, computed distributionally, is also in $L^{2}\Omega^{i+1}(M,g).$
\end{defi}

\begin{defi} The minimal extension $d_{min,i}$; this is given by the graph closure of $d_{i}$ on $C_{0}^{\infty}\Omega^{i}(M)$ respect to the norm of $L^{2}\Omega^{i}(M,g)$, that is,
\begin{equation} \mathcal{D}(d_{min,i})=\{\omega\in L^{2}\Omega^{i}(M,g): \exists\ \{\omega_{j}\}_{j\in J}\subset C^{\infty}_{0}\Omega^{i}(M,g),\ \omega_{j}\rightarrow \omega ,\       d_{i}\omega_{j}\rightarrow \eta\in L^{2}\Omega^{i+1}(M,g)\} 
\end{equation}
and in this case $d_{min,i}\omega=\eta$
\end{defi}

Obviously $\mathcal{D}(d_{min,i})\subset \mathcal{D}(d_{max,i})$. Furthermore, from these definitions, it follows immediately that $$d_{min,i}(\mathcal{D}(d_{min,i}))\subset \mathcal{D}(d_{min,i+1}),\ d_{min,i+1}\circ d_{min,i}=0$$ and that $$d_{max,i}(\mathcal{D}(d_{max,i}))\subset \mathcal{D}(d_{max,i+1}),\ d_{max,i+1}\circ d_{max,i}=0.$$ \\Therefore $(L^{2}\Omega^{*}(M,g),d_{max/min,*})$ are both Hilbert complexes and their cohomology groups are denoted by $H_{2,max/min}^{*}(M,g)$.

Another straightforward but important fact is that the Hilbert complex adjoint of \\$(L^{2}\Omega^{*}(M,g),d_{max/min,*})$ is $(L^{2}\Omega^{*}(M,g),\delta_{min/max,*})$ with $\delta_{*}$ the formal adjoint of $d_{*}$, that is
\begin{equation}
(d_{max,i})^{*}=\delta_{min,i},\     (d_{min,i})^{*}=\delta_{max,i}.
\end{equation}

Using  proposition \ref{kio}  we obtain two weak Kodaira decompositions:
\begin{equation}
L^{2}\Omega^{i}(M,g)=\mathcal{H}^{i}_{abs/rel}\oplus \overline{ran(d_{max/min,i-1})}\oplus \overline{ran(\delta_{min/max,i})}
\label{kd}
\end{equation}
with summands mutually orthogonal in each case. The first summand in the right, called the absolute or relative  Hodge cohomology, respectively, is defined as the orthogonal complement of the other two summands. Since $(ran(d_{max,i-1}))^{\bot}=Ker(\delta_{min,i-1})$ and $(ran(d_{min,i-1}))^{\bot}=Ker(\delta_{max,i-1})$, we see that 
\begin{equation}
\mathcal{H}^{i}_{abs/rel}=Ker(d_{max/min,i})\cap Ker(\delta_{min/max,i-1}).
\end{equation}

Now consider the following operators:
\begin{equation}
\Delta_{abs,i}=\delta_{min,i}d_{max,i}+d_{max,i-1}\delta_{min,i-1},\     \Delta_{rel,i}=\delta_{max,i}d_{min,i}+d_{min,i-1}\delta_{max,i-1}
\end{equation}
These are selfadjoint and satisfy:
\begin{equation}
\mathcal{H}^{i}_{abs}(M,g)=Ker(\Delta_{abs,i}),\ \mathcal{H}^{i}_{rel}(M,g)=Ker(\Delta_{rel,i})
\end{equation} and
\begin{equation}
\overline{ran(\Delta_{abs,i})}=\overline{ran(d_{max,i-1})}\oplus \overline{ran(\delta_{min,i})},\ \overline{ran(\Delta_{rel,i})}=\overline{ ran(d_{min,i-1})}\oplus \overline{ran(\delta_{max,i})}.
\end{equation}
Furthermore, by proposition \ref{qwe}, if $H^{i}_{2,max/min}(M,g)$ is finite dimensional then the range of $d_{max/min,i-1}$ is closed and $\mathcal{H}^{i}_{abs/rel}(M,g)\cong H_{2,max/min}^{i}(M,g)$. On $L^{2}\Omega^{i}(M,g)$ we have also a third weak Koidara decomposition which is the original one considered by Kodaira in \cite{KK}.
\begin{equation}
L^{2}\Omega^{i}(M,g)=\mathcal{H}^{i}_{max}\oplus \overline{ran(d_{min,i-1})}\oplus \overline{ran(\delta_{min,i})}
\label{oiko}
\end{equation}
where $\mathcal{H}^{i}_{max}(M,g)$ is $ Ker(d_{max,i})\cap Ker(\delta_{max,i-1})$ and it is called the $i-th$ maximal Hodge cohomology group.\\ We can also consider the following operators: 
\begin{equation}
\Delta_{max,i}:L^{2}\Omega^{i}(M,g)\rightarrow L^{2}\Omega^{i}(M,g),\ \Delta_{min,i}:L^{2}\Omega^{i}(M,g)\rightarrow L^{2}\Omega^{i}(M,g).
\label{jiu}
\end{equation}
$\Delta_{max,i}$ is defined as the maximal closure of $\delta_{i}\circ d_{i}+d_{i-1}\circ \delta_{i-1}:C^{\infty}_{c}\Omega^{i}(M)\rightarrow C^{\infty}_{c}\Omega^{i}(M)$ that is $u\in \mathcal{D}(\Delta_{max,i})$ and $v=\Delta_{max,i}(u)$ if $$<u,\delta_{i}( d_{i}(\phi))+d_{i-1}( \delta_{i-1}(\phi))>_{L^{2}(M,g)}=<v,\phi>_{L^{2}(M,g)}\ \text{for each } \phi \in C^{\infty}_{c}\Omega^{i}(M).$$ $\Delta_{min,i}$ is the minimal closure of $\delta_{i}\circ d_{i}+d_{i-1}\circ \delta_{i-1}:C^{\infty}_{c}\Omega^{i}(M)\rightarrow C^{\infty}_{c}\Omega^{i}(M)$ that is $u\in \mathcal{D}(\Delta_{min,i})$ and $v=\Delta_{min,i}(u)$ if there is a sequence $\{\phi\}_{i\in \mathbb{N}}\subset C^{\infty}_{c}\Omega^{i}(M)$ such that $$\phi_{i}\rightarrow u\ \text{in}\  L^{2}\Omega^{i}(M,g)\ \text{and}\ \delta_{i}( d_{i}(\phi))+d_{i-1}( \delta_{i-1}(\phi))\rightarrow u\ \text{in}\ L^{2}\Omega^{i}(M,g).$$
\begin{prop} The operators $\Delta_{max,i},\ \Delta_{min,i}$ satisfy the following properties:
\label{olpo}
\begin{enumerate}
 \item $(\Delta_{max,i})^{*}=\Delta_{min,i}, (\Delta_{min,i})^{*}=\Delta_{max,i}.$
\item $Ker(\Delta_{min,i})=Ker(d_{min,i})\cap Ker(\delta_{min,i-1})$. We call it the $i-th$ minimal Hodge cohomology group and we label it $\mathcal{H}^{i}_{min}(M,g).$
\item $Ker(\Delta_{max,i})=Ker(d_{max,i})\cap Ker(\delta_{max,i-1})=\mathcal{H}^{i}_{max}(M,g)$.
\item $\overline{ran(\Delta_{min,i})}=\overline{ran(d_{min,i-1})}\oplus \overline{ran(\delta_{min,i})}.$
\item $\overline{ran(\Delta_{max,i})}=\overline{ran(d_{max,i-1})+ran(\delta_{max,i})}.$
\end{enumerate}
\end{prop}

\begin{proof} The first property is immediate. For the second property consider the following operator: $d_{max,i-1}\circ \delta_{min,i-1}+\delta_{max,i}\circ d_{min,i}:L^{2}\Omega^{i}(M,g)\rightarrow L^{2}\Omega^{i}(M,g)$. We label it $\Delta_{m,i}$. This is a symmetric operator and it is clear that $\Delta_{m,i}$ extends $\Delta_{min,i}$ that is $\mathcal{D}(\Delta_{min,i})\subset \mathcal{D}(\Delta_{m,i})$ and $\Delta_{min,i}(u)=\Delta_{m,i}(u)$ for each $u\in \Delta_{min,i}$. From this it follows that $Ker(\Delta_{min,i})\subset \mathcal{H}^{i}_{min}(M,g)$ because $Ker(\Delta_{min,i})\subset Ker(\Delta_{m,i})$ and $Ker(\Delta_{m,i})=\mathcal{H}^{i}_{min}(M,g)$. By the fact that $ran(\Delta_{max,i})\subset \overline{ran(d_{max,i-1})+ran(\delta_{max,i})}$ and by the first property it follows that $Ker(\Delta_{min,i})=(ran(\Delta_{max,i}))^{\bot}\supset (\overline{ran(d_{max,i-1})+ran(\delta_{max,i})})^{\bot}=\mathcal{H}^{i}_{min}(M,g)$. Therefore $Ker(\Delta_{min,i})=\mathcal{H}^{i}_{min}(M,g)$.\\ For the third property consider the following operator: $d_{min,i-1}\circ \delta_{max,i-1}+\delta_{min,i}\circ d_{max,i}:L^{2}\Omega^{i}(M,g)\rightarrow L^{2}\Omega^{i}(M,g)$. We label it $\Delta_{M,i}$. Also $\Delta_{M,i}$ is a symmetric operator and it is clear that $\Delta_{max,i}$ extends $\Delta_{M,i}$. Therefore  $Ker(\Delta_{max,i})\supset \mathcal{H}^{i}_{max}(M,g)$ because $Ker(\Delta_{max,i})\supset Ker(\Delta_{M,i})$ and $Ker(\Delta_{M,i})=\mathcal{H}^{i}_{max}(M,g)$.\\ Now by the fact that $ran(\Delta_{min,i})\subset \overline{ran(d_{min,i-1})+ran(\delta_{min,i})}$ and by the first property it follows that $Ker(\Delta_{max,i})=(ran(\Delta_{min,i}))^{\bot}\supset (\overline{ran(d_{min,i-1})+ran(\delta_{min,i})})^{\bot}=\mathcal{H}^{i}_{max}(M,g)$. In this way we can conclude that $Ker(\Delta_{max,i})=\mathcal{H}^{i}_{max}(M,g)$.\\ For the fourth property we can observe that $\overline{ran(\Delta_{min,i})}\subset \overline{ran(D_{m,i})}\subset \overline{ran(d_{min,i-1})}\oplus \overline{ran(\delta_{min,i})}$. But, by the third point, $(\overline{ran(d_{min,i-1})}\oplus \overline{ran(\delta_{min,i})})^{\bot}=Ker(\Delta_{max,i})$ and $(Ker(\Delta_{max,i}))^{\bot}=\overline{ran(\Delta_{min,i})}$; therefore the fourth point is proved.\\For the fifth property we can observe that $\overline{ran(\Delta_{max,i})}\subset \overline{ran(d_{max,i-1})+ran(\delta_{max,i})}$. But, by the second  point, $(\overline{ran(d_{max,i-1})+ran(\delta_{max,i})})^{\bot}=Ker(\Delta_{min,i})$ and $(Ker(\Delta_{min,i}))^{\bot}=\overline{ran(\Delta_{max,i})}$ and therefore the fifth point is proved.    
\end{proof}
Finally we conclude the section by stating a result that is a particular case of proposition \ref{dff}.

\begin{prop}[\cite{BL}, pag 110, \cite{C} appendix] Consider the smooth differential forms $\Omega^{*}(M)$ and the following complex:
\begin{equation}
(\Omega^{*}_{2}(M,g),d_{*}):= 0\rightarrow \Omega_{2}^{0}(M,g)\stackrel{d_{0}}{\rightarrow}\Omega_{2}^{1}(M,g)\stackrel{d_{1}}{\rightarrow}...\stackrel{d_{n-1}}{\rightarrow}\Omega_{2}^{n}(M,g)\stackrel{d_{n}}{\rightarrow}0
\label{fei}
\end{equation}
where $\Omega^{i}_{2}(M,g)=\{\omega\in \Omega^{i}(M): \|\omega\|_{L^{2}(M,g)}<\infty\ and\ \|d_{i}\omega\|_{L^{2}(M,g)}<\infty\}$.\\Then $(\Omega^{*}_{2}(M,g),d_{*})$ is a subcomplex quasi-isomorphic to the complex $(L^{2}\Omega^{*}(M,g),d_{max,*})$
\label{zaq}
\end{prop}



\subsection{Stratified pseudomanifolds and intersection homology}

We begin by recalling the concept of stratified pseudomanifold. The definition is given by induction on the dimension.
\begin{defi} A $0-$dimensional stratified space is a countable set with the discrete topology. For $m>0$ a $m-$dimensional topologically stratified space is paracompact Hausdorff topological space  $X$  equipped with a filtration
\begin{equation} X=X_{m}\supset X_{m-1}\supset...\supset X_{1}\supset X_{0}
\label{strati}
\end{equation}
of $X$ by  closed subsets $X_{j}$ such that if $x\in X_{j}-X_{j-1}$ there exists a neighbourhood $N_{x}$ of $x$ in $X$, a compact $(m-j-1)-$dimensional topologically stratified space $L$ with a filtration 
\begin{equation}
L=L_{m-j-1}\supset...\supset L_{1}\supset L_{0}
\label{stratii}
\end{equation}
and a homeomorphism
\begin{equation} \phi:N_{x}\rightarrow \mathbb{R}^{j}\times C(L)
\label{stratiii}
\end{equation}
where $C(L)=L\times [0,1)/L\times \{0\}$ is the open cone on $L$, such that $\phi$ takes $N_{x}\cap X_{j+i+1}$ homeomorphically onto
\begin{equation} \mathbb{R}^{j}\times C(L_{i})\subset \mathbb{R}^{j}\times C(L)
\label{stratum}
\end{equation}
for $m-j-1\geq i\geq 0$ and $\phi$ takes $N_{x}\cap X_{j}$ homeomorphically onto  
\begin{equation}\mathbb{R}^{j}\times \{vertex\ of\ C(L)\}
\label{loi}
\end{equation}
\label{wert}
\end{defi}

This definition guaranties that, for each $j$, the subset $X_{j}-X_{j-1}$ is a topological manifold of dimension $j$. The \textbf{strata} of $X$ are the connected components of these manifolds. If a stratum $Y$ is a subset of $X-X_{n-1}$ it is called a \textbf{regular stratum}; otherwise it is called a \textbf{singular stratum}.  
The space L  is referred as to the \textbf{link} of the stratum. In general it is not uniquely determined up to homeomorphism, though if $X$ is a stratified pseudomanifold it is unique up to stratum preserving homotopy equivalence (see\cite {IGP} pag 108).
\begin{defi} A topological pseudomanifold of dimension $m$ is a paracompact Hausdorff topological space $X$ which posses a topological stratification such that
\begin{equation} X_{m-1}=X_{m-2}
\label{asw}
\end{equation}
and $X-X_{m-2}$ is dense in $X$.(For more details see \cite{BA} or \cite{KW}).
\end{defi}

Over these spaces, at the end of the seventies,  Mark Goresky  and Robert MacPherson have defined a new homological theory known as intersection homology. Here we recall briefly the main definitions and we refer to \cite{BA}, \cite{B}, \cite{GM}, \cite{GMA} and \cite{KW} for a complete development of the theory.
\begin{defi} A perversity is a function $p:\{2,3,4,...,n\}\rightarrow \mathbb{N}$ such that
\begin{equation} p(2)=0\ and\ p(i)\leq p(i+1)\leq p(i)+1.
\end{equation}
\label{per}
\end{defi}

Let $\Delta_{i}\subset \mathbb{R}^{i+1}$ the standard $i-$simplex. The $j-$\textbf{skeleton} are of $\Delta_{i}$ is the set of $j-$subsimplices. We say a singular $i-$simplex in $X$, i.e. a continuous map $\sigma:\Delta_{i}\rightarrow X$, is $p-$\textbf{allowable} if
\begin{equation} \sigma^{-1}(X_{m-k}-X_{m-k-1})\subset \{(i-k+p(k))-skeleton\ of\ \Delta_{i}\}\ for\ all\ k\geq 2.
\end{equation}
The elements of the space $I^{p}S_{i}(X)$ are the finite linear combinations of singular $i-$simplex $\sigma:\Delta_{i}\rightarrow X$ such that $\sigma$ and $\partial \sigma$ are $p-$allowable. Clearly $(I^{p}S_{i}(X), \partial_{i})$ is a complex, more precisely a subcomplex of $( S_{i}(X), \partial_{i})$, and the \textbf{perversity p singular intersection homology groups}, $I^{p}H_{i}(X)$, are the homology groups of this complex.

\begin{rem} The above definition is not the original definition given by Goresky and MacPherson in \cite{GM}. In fact in their paper Goresky and MacPherson use a simplicial point of view and in particular the notion of p-allowable simplicial chains. The definition  that we have recalled here was given in \cite{K} by H. King. Over a PL-stratified pseudomanifold it is equivalent to the Goresky and MacPherson's definition but the advantage is that it holds even if $X$ is only a stratified pseudomanifold.
\end{rem}

However, for our goals we need a more general notion of perversity and associated intersection homology. A generalization of the theory of Goresky and MacPherson that is suited for our needs was made by Greg Friedman. As in the previous case we recall only the main definitions and results and we refer to the \cite{GP}, \cite{IGP} and \cite{SP} for a complete development of the theory.\\First, we remember that the theory proposed by Friedman applies to a \textbf{wider class of spaces}: from now on a \textbf{stratified pseudomanifold will be simply  a paracompact Hausdorff topological space $X$ which posses a topological stratification and such that $X-X_{n-1}$ is dense in $X$}. \textbf{ That is, we do not require  that the condition $X_{m-1}=X_{m-2}$  apply}. In the following propositions each stratified pseudomanifolds will have a\textbf{ fixed stratification}. We start by introducing the notion of \textbf{general perversity}:
\begin{defi} A general perversity on a stratified pseudomanifold $X$ is any function
\begin{equation}
 p:\{Singular\ Strata\ of\ X\}\rightarrow \mathbb{Z}.
\end{equation}
\label{ujn}
\end{defi}

The notion of $p-$\textbf{allowable} singular simplex is modified in the following way:
a singular $i-$simplex in $X$, i.e. a continuous map $\sigma:\Delta_{i}\rightarrow X$, is $p-$\textbf{allowable} if
\begin{equation} \sigma^{-1}(Y)\subset \{(i-cod(Y)+p(Y))-skeleton\ of\ \Delta_{i}\}\ for\ any\ singular\ stratum\ Y\ of\ X.
\end{equation}

A key ingredient in this new theory is the notion of \textbf{homology with stratified  coefficient system}. (The definition uses the notion homology with local coefficient system; for the definition of local coefficient system see \cite{AH}, \cite{NS}, \cite{DK})

\begin{defi} Let $X$ stratified pseudomanifold and let $\mathcal{G}$ a local system on $X-X_{n-1}$. Then the stratified  coefficient sistem $\mathcal{G}_{0}$ is defined to consist of the pair of coefficient systems given by $\mathcal{G}$ on $X-X_{n-1}$ and the constant $0$ system on $X_{n-1}$ i.e. we think of $\mathcal{G}_{0}$ as consisting of a locally constant fiber bundle $\mathcal{G}_{X-X_{n-1}}$ over $X-X_{n-1}$ with fiber $G$ with the discrete topology together with the trivial bundle on $X_{n-1}$ with the stalk $0.$
\label{sc}
\end{defi}

Then a \textbf{coefficient} $n$ of a singular simplex $\sigma$ can be described by a lift of $\sigma|_{\sigma^{-1}(X-X_{n-1})}$ to $\mathcal{G}$ over $X-X_{n-1}$ together with the trivial lift of $\sigma|_{\sigma^{-1}(X_{n-1})}$ to the $0$ system on $X_{n-1}.$ A coefficient  of a simplex $\sigma$ is considered to be the $0$ coefficient if it maps each  points of $\Delta$ to the $0$ section of one of the coefficient systems. Note that if $\sigma^{-1}(X-X_{n-1})$ is path-connected then a coefficient lift of $\sigma$ to $\mathcal{G}_{0}$ is completely determined by the lift at a single point of $\sigma^{-1}(X-X_{n-1})$ by the lifting extension property for $\mathcal{G}$. The intersection homology chain complex $(I^{p}S_{*}(X,\mathcal{G}_{0}),\partial_{*})$ are defined in the same way as $I^{p}S_{*}(X,G)$, where $G$ is any field, but replacing the coefficient of simplices with coefficient in $\mathcal{G}_{0}$.
 If $n\sigma$ is a simplex $\sigma$ with its coefficient $n$, its boundary is given by the usual formula $\partial(n\sigma)=\sum_{j}(-1)^{j}(n\circ i_{j})(\sigma\circ i_{j})$ where $i_{j}:\Delta_{i-1}\rightarrow \Delta_{i}$ is the $j-$face inclusion map. Here $n\circ i_{j}$ should be interpreted as the restriction of $n$ to the $jth$ face of $\sigma$, restricting the lift to $\mathcal{G}$ where possible and restricting to $0$ otherwise. The basic idea behind the definition is  that when we consider if a chain is allowable with respect to a perversity, simplices with support entirely in $X_{n-1}$ should vanish and thus not be counted for admissibility considerations. (For more details see \cite{GP}, \cite{IGP} and \cite{SP}).\vspace{1 cm}
 
The next proposition shows that Friedman's theory is an extension of the classical theory made by Goresky and MacPherson. 
 
\begin{prop} (see \cite{IGP} pag. 110, \cite{SP} pag. 1985) If $p$ is a traditional perversity, that is a perversity like those defined in definition \ref{per}, and $X_{n-1}=X_{n-2}$  then $$I^{p}S_{*}(X,\mathcal{G})=I^{p}S_{*}(X,\mathcal{G}_{0}).$$
\label{wik}
\end{prop}

\begin{exa}
Let $X$ be a stratified pseudomanifold and $p$ a general perversity on $X$. Consider as stratified coefficient system $\mathcal{R}_{0}$, that is  the pair of coefficient systems given by $(X-X_{n-1})\times \mathbb{R}$ over $X-X_{n-1}$ where the fibers $\mathbb{R}$ have the discrete topology  and the constant $0$ system on $X_{n-1}$. Now suppose that $X$ and $p$ satisfy the assumptions of proposition \ref{wik}; then $$I^{p}S_{*}(X,\mathbb{R})=I^{p}S_{*}(X,\mathcal{R}_{0})$$ where $I^{p}S_{*}(X,\mathbb{R})$ is the usual intersection homology chain complex with coefficient in the field $\mathbb{R}$. 
\label{ytr}
\end{exa}

We conclude this section  recalling  some fundamental results of this theory that generalize the previous results obtained by Goresky and MacPherson.\\

Let $X$ a stratified pseudomanifold, $\mathfrak{X}$ a fixed stratification on $X$, $p$ a generalized perversity on $X$, $\mathcal{G}$ a local system on $X-X_{n-1}$ and $\mathcal{O}$ the orientation sheaf on $X-X_{n-1}$.\\ Consider now the following \textbf{set of axioms} $(AX1)_{p, \mathfrak{X},\mathcal{G}\otimes \mathcal{O}}$ for a complex of sheaves $(\mathcal{S}^{*}, d_{*})$:
\begin{enumerate}
\item $\mathcal{S}^{*}$\ is\ bounded, $\mathcal{S}^{i}=0$\ for $i<0$ and\ $\mathcal{S}^{*}|_{X-X_{n-1}}$ is quasi-isomorphic to $\mathcal{G}\otimes \mathcal{O}.$
\item If $x\in Z$ for a stratum $Z$, then $H_{i}(\mathcal{S}_{x}^{*})=0$ for $i>p(Z).$
\item Let $U_{k}=X-X_{n-k}$ and let $i_{k}:U_{k}\rightarrow U_{k+1}$ the natural inclusion. Then for $x\in Z\subset U_{k+1}$  the attachment map $\alpha_{k}:\mathcal{S}^{*}|_{U_{k+1}}\rightarrow Ri_{k*}i_{k}^{*}\mathcal{S}^{*}|_{U_{k+1}}$, given by the composition of natural morphism $\mathcal{S}^{*}|_{U_{k+1}}\rightarrow i_{k*}i_{k}^{*}\mathcal{S}^{*}|_{U_{k+1}}\rightarrow Ri_{k*}i_{k}^{*}\mathcal{S}^{*}|_{U_{k+1}}$, is a quasi-isomorphism at $x$ up to $p(Z).$
\label{ax}
\end{enumerate}

In almost all references the previous axioms are formulated in the derived category of sheaves on $X$. In that case the term quasi-isomorphism should be replaced with the term isomorphism.

\begin{teo} (see \cite{GP} pag 116) Let $X$ a compact stratified pseudomanifold of dimension $n$, $p$ a general perversity on $X$ and $(\mathcal{S}^{*}, d_{*})$ a complex of sheaves that satisfies the set of axioms $(AX1)_{p, \mathfrak{X},\mathcal{G}\otimes \mathcal{O}}$. Then the following isomorphism holds: 
\begin{equation}
\mathbb{H}^{i}(X,\mathcal{S}^{*})\cong I^{p}H_{n-i}(X,\mathcal{G}_{0})
\end{equation} 
that is the $i-$th hypercohomology group of the complex $(\mathcal{S}^{*}, d_{*})$ is isomorphic to the $(n-i)-$th intersection homology group with coefficient in the stratified system $\mathcal{G}_{0}$ and relative to the perversity $p$. 
\label{tre}
\end{teo}

\begin{cor} In the same hypothesis of the previous theorem if $(\mathcal{S}^{*}, d_{*})$ is a complex of fine or flabby or soft sheaves then the following isomorphism holds:
\begin{equation}
H^{i}(\mathcal{S}^{*}(X),d_{*})\cong I^{p}H_{n-i}(X,\mathcal{G}_{0})
\end{equation}
where $H^{i}(\mathcal{S}^{*}(X),d_{*})$ are the cohomology groups of the complex $$0...\stackrel{d_{i-1}}{\rightarrow}\mathcal{S}^{i}(X)\stackrel{d_{i}}{\rightarrow}\mathcal{S}^{i+1}(X)\stackrel{d_{i+1}}{\rightarrow}\mathcal{S}^{i+2}(X)\stackrel{d_{i+2}}{\rightarrow}...$$
\label{mlp}
\end{cor}

\begin{teo}(see \cite{GP} pag 122 or \cite{IGP} pag 25.) Let $F$ a field, $X$ a compact and $F-$oriented stratified pseudomanifold of dimension $n$, $p,\ q$  general perversities on $X$ such that $p+q=t$ (that is for each stratum $Z\subset X$ $p(Z)+q(Z)=codim(Z)-2$) and $\mathcal{F}_{0}$ a stratified coefficient system over $X$,  consisting of  the pair of coefficient systems given by $(X-X_{n-1})\times F$ over $X-X_{n-1}$ where the fibers $F$ have the discrete topology  and the constant $0$ system on $X_{n-1}$. Then the following isomorphism holds:
\begin{equation}
I^{p}H_{i}(X,\mathcal{F}_{0})\cong Hom(I^{q}H_{n-i}(X,\mathcal{F}_{0}), F).
\end{equation}
\label{mzb}
\end{teo}

\begin{rem} In this paper with the symbol $I^{p}H^{i}(X,\mathcal{G}_{0})$ we mean the cohomology of the complex $$(Hom(I^{p}S_{i}(X,\mathcal{G}_{0}), G), (\partial_{i})^{*}).$$ We call it the $i-th$ intersection cohomology group of $X$ with respect to the perversity $p$ and the stratified coefficient system $\mathcal{G}_{0}$. When $G=F$ is a field then $$I^{p}H^{i}(X,\mathcal{F}_{0})\cong Hom(I^{p}H_{i}(X,\mathcal{F}_{0}), F).$$
\end{rem}

\begin{rem} Summarizing, by theorems \ref{tre} and \ref{mzb}, it  follows that if $(\mathcal{S}^{*}, d_{*})$ is a complex of sheaves that satisfies the set of axioms $(AX1)_{p, \mathfrak{X},\mathcal{F}\otimes \mathcal{O}}$ then
\begin{equation}
\mathbb{H}^{i}(X,\mathcal{S}^{*})\cong I^{q}H^{i}(X,\mathcal{F}_{0})
\end{equation} 
where $p+q=t$ and if  $(\mathcal{S}^{*}, d_{*})$ is a complex of fine or flabby or soft sheaves then, by corollary \ref{mlp},
\begin{equation}
H^{i}(\mathcal{S}^{*}(X),d_{*})\cong I^{q}H^{i}(X,\mathcal{F}_{0})
\end{equation}
\label{ooo}
\end{rem}

\subsection{Thom-Mather stratification and  quasi edge metrics with weights}

We start this subsection by giving the definition  of a smoothly stratified pseudomanifold with a Thom-Mather stratification. We follow \cite{ALMP}.
\begin{defi}   
 A smoothly stratified pseudomanifold $X$ with a Thom-Mather stratification is a metrizable, locally compact, second countable space which admits a locally finite decomposition into a union of locally closed strata $\mathfrak{G}=\{Y_{\alpha}\}$, where each $Y_{\alpha}$ is a smooth, open and connected manifold, with dimension depending on the index $\alpha$. We assume the following:
\begin{enumerate}
\item If $Y_{\alpha}$, $Y_{\beta} \in \mathfrak{G}$ and $Y_{\alpha} \cap \overline{Y}_{\beta} \neq \emptyset$ then $Y_{\alpha} \subset \overline{Y}_{\beta}$
\item  Each stratum $Y$ is endowed with a set of control data $T_{Y} , \pi_{Y}$ and $\rho_{Y}$ ; here $T_{Y}$ is a neighbourhood of $Y$ in $X$ which retracts onto $Y$, $\pi_{Y} : T_{Y} \rightarrow Y$
is a fixed continuous retraction and $\rho_{Y}: T_{Y}\rightarrow [0, 2)$ is a proper radial function in this tubular neighbourhood such that $\rho_{Y}^{-1}(0) = Y$ . Furthermore,
we require that if $Z \in \mathfrak{G}$ and $Z \cap T_{Y}\neq \emptyset$  then
$(\pi_{Y} , \rho_{Y} ) : T_{Y} \cap Z \rightarrow Y \times [0, 2)$
is a proper differentiable submersion.
\item If $W, Y,Z \in \mathfrak{G}$, and if $p \in T_{Y} \cap T_{Z} \cap W$ and $\pi_{Z}(p) \in T_{Y} \cap Z$ then
$\pi_{Y} (\pi_{Z}(p)) = \pi_{Y} (p)$ and $\rho_{Y} (\pi_{Z}(p)) = \rho_{Y} (p)$.
\item If $Y,Z \in \mathfrak{G}$, then
$Y \cap \overline{Z} \neq \emptyset \Leftrightarrow T_{Y} \cap Z \neq \emptyset$ ,
$T_{Y} \cap T_{Z} \neq \emptyset \Leftrightarrow Y\subset \overline{Z}, Y = Z\ or\ Z\subset \overline{Y} .$
\item  For each $Y \in \mathfrak {G}$, the restriction $\pi_{Y} : T_{Y}\rightarrow Y$ is a locally trivial fibration with fibre the cone $C(L_{Y})$ over some other stratified space $L_{Y}$ (called the link over $Y$ ), with atlas $\mathcal{U}_{Y} = \{(\phi,\mathcal{U})\}$ where each $\phi$ is a trivialization
$\pi^{-1}_{Y} (U) \rightarrow U \times C(L_{Y} )$, and the transition functions are stratified isomorphisms  which preserve the rays of each conic
fibre as well as the radial variable $\rho_{Y}$ itself, hence are suspensions of isomorphisms of each link $L_{Y}$ which vary smoothly with the variable $y\in U$.
\item For each $j$ let $X_{j}$ be the union of all strata of dimension less or equal than $j$, then $$X-X_{n-1}\ is\ dense\ in\ X$$
\end{enumerate}
\label{thom}
\end{defi}


We make a few comments to the previous definition (for more details we refer to \cite{ALMP}):
\begin{enumerate}
\item The previous definition is more general than that given in \cite{ALMP}. In \cite{ALMP} a space that satisfies the definition \ref{thom} is only a smoothly stratified spaces (with a Thom-Mather stratification). To be a smoothly stratified pseudomanifold (with a Thom-Mather stratification) there is another requirement to satisfy: let $Xj$ be the union of all strata of dimensions less  or equal
than $j$, then
\begin{equation}X = X_{n}\supset X_{n-1} = X_{n-2}\supset X_{n-3}\supset...\supset X_{0}
\label{fil}
\end{equation} and $X-X_{n-2}$ is dense in $X.$ For our goals, thanks to the results of Friedman, we can waive the requirement $X_{n-1}=X_{n-2}$ and therefore we will call  smoothly stratified pseudomanifold with a Thom-Mather stratification each space $X$ that satisfies the definition \ref{thom}.
\item The link $L_{Y}$ is uniquely determined, up to isomorphism (see point number 5 below for the notion of isomorphism), by the stratum $Y$. 
\item The depth of a stratum $Y$ is largest integer $k$ such that there is a chain of strata $Y=Y_{k},...,Y_{0}$ such that $Y_{j}\subset \overline{Y_{j-1}}$ for $i\leq j\leq k.$ A stratum of maximal depth is always a closed subset of $X$.  The  maximal depth of any stratum in $X$ is called the depth of $X$ as stratified spaces.
\item Consider the filtration
\begin{equation}
X = X_{n}\supset X_{n-1}\supset X_{n-2}\supset X_{n-3}\supset...\supset X_{0}
\label{pippo}
\end{equation} 
 We refer to the open subset $X-X_{n-1}$ of a stratified pseudomanifold $X$ as its regular set, and the union of all other strata as the singular set,
$$reg(X):=X-sing(X)\ \text{where}\ sing(X):=\bigcup_{Y\in \mathfrak{G}, depthY>0 }Y. $$
\item If $X, X'$ are two stratified spaces a stratified isomorphism between them is a homeorphism $F:X\rightarrow X'$ which carries the  strata of $X$ to the  strata of $X'$ diffeomorphically, and such that $\pi'_{F(Y)}\circ F=F\circ \pi_{Y},\ \rho_{Y}=\rho'_{(F(Y))}\circ F$ for all $Y\in \mathcal{G}(X).$
\end{enumerate}

Summarizing a smoothly stratified pseudomanifold with Thom-Mather stratification is a  stratified pseudomanifold with a richer structure from a differentiable and topological point of view.\vspace{1 cm}

Now we introduce an important class of riemannian metrics on the regular part of a smoothly stratified pseudomanifold  with a Thom-Mather stratification. Before giving the definition we recall that two riemannian metrics $g,h$ on a smooth manifold $M$ are  \textbf{quasi-isometric} if there are constants $c_{1}, c_{2}$ such that $c_{1}h\leq g\leq c_{2}h$. 

\begin{defi} Let $X$ be a smoothly stratified pseudomanifold with a Thom-Mather stratification and let $g$ a riemannian metric on $reg(X)$.     We call  $g$ a \textbf{ quasi edge metric with weights} if it satisfies the following properties:
\begin{enumerate}
\item Take any stratum $Y$ of $X$; by   definition \ref{thom} for each $q\in Y$ there exist an open neighbourhood $U$ of $q$ in $Y$ such that $\phi:\pi_{Y}^{-1}(U)\rightarrow U\times C(L_{Y})$ is a stratified isomorphism; in particular $\phi:\pi_{Y}^{-1}(U)\cap reg(X)\rightarrow U\times reg(C(L_{Y}))$ is a diffeomorphism. Then, for each $q\in Y$, there exists one of these trivializations $(\phi,U)$ such that $g$ restricted on $\pi_{Y}^{-1}(U)\cap reg(X)$ satisfies the following properties:
\begin{equation} 
(\phi^{-1})^{*}(g|_{\pi_{Y}^{-1}(U)\cap reg(X)})\cong dr\otimes dr+h_{U}+r^{2c}g_{L_{Y}}
\label{yhnn}
\end{equation}
 where $h_{U}$ is a riemannian metric defined over $U$, $c\in \mathbb{R}$ and $c>0$, $g_{L_{Y}}$ is a riemannian metric   on $reg(L_{Y})$, $dr\otimes dr+h_{U}+r^{2c}g_{L_{Y}}$ is a riemannian metric of product type on $U\times reg(C(L_{Y}))$ and with $\cong$ we mean \textbf{quasi-isometric}. 
 \item  If $p$ and $q$ lie in the same stratum $Y$ then in \eqref{yhnn} there is the same weight. We label it $c_{Y}$. 
\end{enumerate}
\label{zedge}
\end{defi}

Before continuing we make some \textbf{remarks}:

\begin{enumerate}
\item Obviously if the codimension of $Y$ is $1$ then $L_{Y}$ is just a point and therefore by the previous definition $(\phi^{-1})^{*}(g|_{\pi_{Y}^{-1}(U)\cap reg(X)})\cong dr\otimes dr+h_{U}$.
\item In the first point of the previous definition the metric $g_{L_{Y}}$ depends also on the open neighborhood $U$ and the stratified isomorphism $\phi$. However we prefer to use the notation $g_{L_{Y}}$ instead of $g_{L_{Y},U,\phi}$ for the sake of  simplicity.
\item Let $g$ and $U$ be like in the first point of the previous definition and let $\psi:\pi_{Y}^{-1}(U)\rightarrow U\times C(L_{Y})$ another stratified isomorphism that satisfies the requirements of definition \ref{thom}. From the fifth point of definition \ref{thom} it follows that $\psi\circ \phi^{-1}:U\times C(L_{Y})\rightarrow U\times C(L_{Y})$ acts in this way: given $p=(y,[r,x])\in U\times C(L_{Y})$ $(\psi\circ \phi^{-1})(p)=(y,[r,f(y,x)])$ where the maps $x\mapsto f(y,x)$ are a family of smooth stratified isomorphisms of $L_{Y}$  which vary smoothly with the variable $y\in U$. From this it follows immediately that if we fix a point $y_{0}\in U$  and if we put $h_{L_{Y}}=(f(y_{0},x)^{-1})^{*}(g_{L_{Y}})$ then there exists an open subset $ V\subset U, y_{0}\in V$  such that $(\psi^{-1})^{*}(g|_{\pi_{Y}^{-1}(V)\cap reg(X)})\cong dr\otimes dr+h_{U}|_{V}+r^{2c_{Y}}h_{L_{Y}}$ where $h_{U}|_{V}$ is the metric $h_{U}$ restricted to $V$. Therefore the weight $c_{Y}$ does not depend from the particular trivialization $\phi$ that it is chosen.
\end{enumerate}

Now we give a definition which is a more refined version of the previous one; it is also a slight generalization of the definition of the adapted metric given by Brasselet, Hector and Saralegi in \cite{BHS}.  This definition  is given by induction on $depth(X)$.

\begin{defi} Let $X$ be a stratified pseudomanifold with a Thom-Mather stratification and let $g$ a riemannian metric on $reg(X)$. If $depth(X)=0$, that is $X$ is a closed manifold, a \textbf{ quasi rigid iterated  edge metric with weights} is any riemannian metric on $X$. Suppose now that $depth(X)=k$ and that the definition of quasi rigid iterated  edge metric with weights  is given in the case $depth(X)\leq k-1$; then   we call a riemannian metric $g$ on $reg(X)$ a \textbf{quasi rigid iterated  edge metric with weights} if it satisfies the following properties:
\begin{enumerate}
\item Take any stratum $Y$ of $X$; by   definition \ref{thom} for each $q\in Y$ there exist an open neighbourhood $U$ of $q$ in $Y$ such that $\phi:\pi_{Y}^{-1}(U)\rightarrow U\times C(L_{Y})$ is a stratified isomorphism; in particular $\phi:\pi_{Y}^{-1}(U)\cap reg(X)\rightarrow U\times reg(C(L_{Y}))$ is a diffeomorphism. Then, for each $q\in Y$, there exists one of these trivializations $(\phi,U)$ such that $g$ restricted on $\pi_{Y}^{-1}(U)\cap reg(X)$ satisfies the following properties:
\begin{equation} 
(\phi^{-1})^{*}(g|_{\pi_{Y}^{-1}(U)\cap reg(X)})\cong dr\otimes dr+h_{U}+r^{2c}g_{L_{Y}}
\label{yhn}
\end{equation}
 where $h_{U}$ is a riemannian metric defined over $U$, $c\in \mathbb{R}$ and $c>0$, $g_{L_{Y}}$ is a \textbf{quasi rigid iterated edge metric with weights}  on $reg(L_{Y})$, $dr\otimes dr+h_{U}+r^{2c}g_{L_{Y}}$ is a riemannian metric of product type on $U\times reg(C(L_{Y}))$ and with $\cong$ we mean \textbf{quasi-isometric}. 
 \item If $p$ and $q$ lie in the same stratum $Y$ then in \eqref{yhn} there is the same weight. We label it $c_{Y}$.
\end{enumerate}
\label{edge}
\end{defi}

 Also in this case  a remark to the previous definition is in order. Let $\psi:\pi_{Y}^{-1}(U)\rightarrow U\times C(L_{Y})$ another stratified isomorphism that satisfies the requirements of definition \ref{thom}.  Using the same observations and notations of the second remark of definition \ref{zedge} we can conclude that there exists an open subset $V\subset U$ and a quasi rigid iterated edge metric with weights $h_{L_{Y}}$ on $reg(L_{Y})$ such that $(\psi^{-1})^{*}(g|_{\pi_{Y}^{-1}(V)\cap reg(X)})\cong dr\otimes dr+h_{U}|_{V}+r^{2c_{Y}}h_{L_{Y}}.$ Furthermore, by the fact that $f(y_{0},x)$ is a smooth stratified isomorphism between $L_{Y}$ and $L_{Y}$ such that $(f(y_{0},x))^{*}(h_{L_{Y}})=g_{L_{Y}}$, it follows that $g_{L_{Y}}$ and $h_{L_{Y}}$ have the same weights and therefore, by proposition \ref{tip} below, $g_{L_{Y}}$ and $h_{L_{Y}}$ are quasi-isometric on $reg(L_{Y})$ when $L_{Y}$ is compact.
\begin{prop} Let $X$ be a smoothly stratified pseudomanifold with a Thom-Mather stratification $\mathfrak{X}$. For any stratum $Y\subset X$ fix a positive real number $c_{Y}$. Then there exists a quasi rigid iterated  edge metric with weights $g$ on $reg(X)$ having the numbers $\{c_{Y}\}_{Y\in \mathfrak{X}}$ as weights.
\label{top}
\end{prop}

\begin{proof}
In \cite{ALMP} is defined a class of riemannian metric called \textbf{rigid iterated edge metric} and in prop. 3.1 of the same paper is proved the existence of such metrics. Using the same notation of definition \ref{edge} a riemannian metric $g$ on $reg(X)$ is a rigid iterated edge metric if $(\phi^{-1})^{*}(g|_{\pi_{Y}^{-1}(U)\cap reg(X)})= dr\otimes dr+h_{U}+r^{2}g_{L_{Y}}(u,y)$, with $u\in U$, $y\in L_{Y}$, and for any fixed $u$,  $g_{L_{Y}}(u,y)$ is a rigid iterated edge metric on $reg(L_{Y})$. In \cite{ALMP}  proposition 3.1 is proved in the case $X_{n-1}=X_{n-2}$ but  it is easy to see that it holds also in our case that is when $X_{n-1}\neq X_{n-2}$ and $c_{Y}\neq 1$ . Therefore on $reg(X)$ there is a rigid iterated edge metric $g$ having the numbers $\{c_{Y}\}_{Y\in \mathfrak{X}}$ as weights. Using again the notation of definition \ref{edge} this means that 
 for each stratum $Y$ and for any point $q\in Y$ $(\phi^{-1})^{*}(g|_{\pi_{Y}^{-1}(U)\cap reg(X)})= dr\otimes dr+h_{U}+r^{2c_{Y}}g_{L_{Y}}(u,y)$, with $u\in U$, $y\in L_{Y}$, and for any fixed $u$,  $g_{L_{Y}}(u,y)$ is a rigid iterated edge metric with weights on $reg(L_{Y})$.  Now it is clear that $g$ 
 is a quasi rigid iterated edge metric on $reg(X)$ having the numbers $\{c_{Y}\}_{Y\in \mathfrak{X}}$ as weights. Alternatively the existence of such metrics follows using the same arguments used by Brasselet, Hector and Saralegi in \cite{BHS}.
 
\end{proof}

\begin{prop} Let $X$ be  a compact smoothly stratified pseudomanifold with a Thom-Mather stratification. For any stratum $Y\subset X$ fix a positive real number $c_{Y}$. Let $g,g'$ two quasi edge metrics with weights on $reg(X)$ having both the numbers $\{c_{Y}\}_{Y\in \mathfrak{X}}$ as weights. Then $g$ and $g'$ are quasi-isometric.
\label{tip}
\end{prop}

\begin{proof}
Let $K$ be a compact subset of $X$ such that $K\subset reg(X)$. Obviously $g|_{K}$ is quasi-isometric to $g'|_{K}$. Now let $Y$ be a stratum such that $Y\subset X_{n-1}-X_{n-2}$. Let $x\in Y$; consider $\pi_{Y}^{-1}(x)$ and let $V_{Y,x}:=\pi_{Y}^{-1}(x)\cap \rho_{Y}^{-1}(1)$. Then there exists a compact subset of $X,\ K$ such that $K\subset reg(X)$ and $reg(V_{Y,x})\subset K$. Therefore $g|_{reg(V_{Y,x})}$ is quasi-isometric to $g'|_{reg(V_{Y,x})}$ and from this it follows that, given an open neighbourhood $U$ of $x$ in $Y$ sufficiently small such that $\pi_{Y}^{-1}(U)\cong U\times C(L_{Y})$, $g|_{reg(\pi_{Y}^{-1}(U))}$ is quasi-isometric to $g'|_{reg(\pi_{Y}^{-1}(U))}$. This last assertion is a consequence of the fact that, by  definition \ref{zedge} and  remarks following it,  there is an isomorphism $\phi:\pi_{Y}^{-1}(U)\rightarrow U\times C(L_{Y})$ such that, by definition \ref{zedge}, $(\phi^{-1})^{*}(g|_{reg(\pi_{Y}^{-1}(U))})$ is quasi isometric to $h+dr^{2}+r^{2c_{Y}}g_{L_{Y}}$ and analogously $(\phi^{-1})^{*}(g'|_{reg(\pi_{Y}^{-1}(U))})$ is quasi isometric to $h'+dr^{2}+r^{2c_{Y}}g'_{L_{Y}}$. But from the fact that $g|_{reg(V_{Y,x})}$ is quasi-isometric to $g'|_{reg(V_{Y,x})}$ it follows that $g_{L_{Y}}$ is quasi-isometric to $g'_{L_{Y}}$ and therefore for a sufficiently small $U$ we get  $g|_{reg(\pi_{Y}^{-1}(U))}$ is quasi-isometric to $g'|_{reg(\pi_{Y}^{-1}(U))}$.
 So we can conclude that if $K\subset (X-X_{n-2})$ is a compact subset then $g|_{reg(K)}$ is quasi-isometric to $g'|_{reg(K)}$. Now consider a stratum  $Z\subset X_{n-2}-X_{n-3}$   and let $x\in Z$. As before consider $\pi_{Z}^{-1}(x)$ and let $V_{Z,x}=\pi_{Z}^{-1}(x)\cap \rho_{Z}^{-1}(1)$. Then there exists a compact subset $K\subset (X-X_{n-2})$ such that $V_{Z,x}\subset K$. From this it follows that $g|_{reg(V_{Z,x})}$ is quasi-isometric to $g'|_{reg(V_{Z,x})}$ and now, as before, we can conclude that given an open neighbourhood $U$ of $x$ in $Z$ sufficiently small such that $g|_{\pi_{Z}^{-1}(U)}\cong U\times C(L_{Z})$, $g|_{reg(\pi_{Z}^{-1}(U))}$ is quasi-isometric to $g'|_{reg(\pi_{Z}^{-1}(U))}$. As before from this it follows that if $K\subset (X-X_{n-3})$ is a compact subset then $g|_{reg(K)}$ is quasi-isometric to $g'|_{reg(K)}$. Now it is obvious that iterating this procedure we obtain what was asserted.
\end{proof}

\begin{cor} Let $X$ be a compact smoothly stratified pseudomanifold with a Thom-Mather stratification and let $g$ a quasi edge metric with weights on $reg(X)$. Then there exist $g'$, a \textbf{ quasi rigid iterated  edge metric with weights} on $reg(X)$, that is  quasi-isometric to $g$.
\label{trucco}
\end{cor}

We conclude this section  introducing the notion of  general perversity associated to a quasi edge  metric with weights.

\begin{defi} Let $X$ be a smoothly stratified pseudomanifold with a Thom-Mather stratification and let $g$ a quasi edge metric with weights on $reg(X)$. Then the general perversity $p_{g}$ associated to $g$ is:
\begin{equation}p_{g}(Y):= Y\longmapsto [[\frac{l_{Y}}{2}+\frac{1}{2c_{Y}}]]= \left\{
\begin{array}{lll}
0 & l_{Y}=0 \\
\frac{l_{Y}}{2}+[[\frac{1}{2c_{Y}}]] & l_{Y}\ even\ and\ l_{Y}\neq 0 \\
\frac{l_{Y}-1}{2}+[[\frac{1}{2}+\frac{1}{2c_{Y}}]]& l_{Y}\ odd
\end{array}
\right.
\end{equation}
where $l_{Y}=dimL_{Y}$ and, given any real and positive number $x$, $[[x]]$ is the greatest integer strictly less than $x$.
\label{pim}
\end{defi}

\section{Preliminary propositions}

 In this  section we follow, with some modifications, \cite{C}. Given an oriented riemannian manifold $(F,g)$ of dimension $f$, $C^{*}(F)$ will be the regular part of $C(F)$, that is $C(F)-\{v\}$, and $g_{c}$ will be the riemannian metric on $C^{*}(F)$ 
\begin{equation} 
\label{fcc}
g_{c}=dr\otimes dr+r^{2c}\pi^{*}g
\end{equation}
 where $\pi:C^{*}(F)\rightarrow F$ is the projection over $F$ and $c\in \mathbb{R},\ c>0$.\\ With the symbol $d_{F}:\Omega^{i}(C^{*}(F))\rightarrow \Omega^{i+1}(C^{*}(F))$ we mean the exterior differential obtained by ignoring the variable $r$. 

\begin{prop} Let $\phi\in L^{2}\Omega^{i}(F,g), \phi \neq 0$ and let $\pi:C^{*}(F)\longrightarrow F$ be the projection. Then $\pi^{*}(\phi)\in L^{2}\Omega^{i}(C^{*}(F),g_{c})$ if and only if $i<\frac{f}{2}+\frac{1}{2c}$. In this case the pullback map is also bounded.
\label{lop}
\end{prop} 

\begin{proof} If $\phi\in L^{2}\Omega^{i}(F,g)$ then $$\|\pi^{*}(\phi)\|_{L^{2}(C^{*}(F),g_{c})}^{2}=\int_{C^{*}(F)}\|\pi^{*}(\phi)\|_{C^{*}(F)}^{2}dvol_{C^{*}(F)}=\int_{0}^{1}\int_{F}r^{c(f-2i)}\|\phi\|_{F}^{2}dvol_{F}dr$$ $$=\|\phi\|_{L^{2}(F,g)}^{2}\int_{0}^{1}r^{c(f-2i)}dr<\infty$$
if and only if $i<\frac{f}{2}+\frac{1}{2c}$. Since $\int_{0}^{1}r^{c(f-2i)}dr$ is independent of $\phi$, the pullback map is bounded.
\end{proof} 

\begin{prop} There exists a constant $K>0$ such that for all $\alpha=\phi+dr\wedge \omega\in L^{2}\Omega^{i}(C^{*}(F),g_{c})$ and for any null set $S\subset (1/2,1)$ there is an $a\in (1/2,1)-S$ such that $$\|\phi(a)\|_{L^{2}(F,g)}^{2}\leq K\|\phi\|_{L^{2}(C^{*}(F),g_{c})}^{2}\leq K\|\alpha\|_{L^{2}(C^{*}(F),g_{c})}^{2}.$$
\label{yui}
\end{prop} 

\begin{proof} Suppose that this proposition is false. Then for any $K>0$ there is a form $\phi\in L^{2}\Omega^{i}(C^{*}(F),g_{c})$ such that $$\|\phi\|_{L^{2}(C^{*}(F),g_{c})}^{2}\geq \int_{1/2}^{1}\int_{F}r^{c(f-2i)}\|\phi\|_{F}^{2}dvol_{F}dr= \int_{1/2}^{1}r^{c(f-2i)}\|\phi(r)\|_{L^{2}(F,g)}^{2}dr >$$ $$ K\|\phi\|_{L^{2}(C^{*}(F),g_{c})}^{2}\int_{(1/2,1)-S}^{1}r^{c(f-2i)}dr=K\|\phi\|_{L^{2}(C^{*}(F),g_{c})}^{2}\int_{(1/2,1)}^{1}r^{c(f-2i)}dr.$$ In this way by choosing $K>(\int_{(1/2,1)}^{1}r^{c(f-2i)}dr)^{-1}$ we obtain a contradiction. 
\end{proof}  

\begin{prop} If $i<\frac{f}{2}+\frac{1}{2c}+1$ and $\alpha= \phi+dr\wedge \omega\in L^{2}\Omega^{i}(C^{*}(F),g_{c})$, then for any $a\in (1/2,1)$ $$K_{a}(\alpha)=\int_{a}^{r}\omega(s)ds\in L^{2}\Omega^{i-1}(C^{*}(F),g_{c})$$ and $K_{a}$ is a bounded operator uniformly in $a\in (1/2,1).$
\label{qaz}
\end{prop} 

\begin{proof} By definition $$\|K_{a}(\alpha)\|_{L^{2}(C^{*}(F),g_{c})}^{2}=\|\int_{a}^{r}\omega(s)ds\|_{L^{2}(C^{*}(F),g_{c})}^{2}=\int_{0}^{1}\int_{F}\|\int_{a}^{r}\omega(s)ds\|_{F}^{2}r^{c(f-2i+2)}dvol_{F}dr.$$
We consider the term $\|\int_{a}^{r}\omega(s)ds\|_{F}^{2}$. 
The following inequality holds :$$\|\int_{a}^{r}\omega(s)ds\|_{F}^{2}\leq(\int_{a}^{r}\|\omega(s)ds\|_{F})^{2}$$ and  using the Schwartz inequalities the right side of this becomes:$$(\int_{a}^{r}\|\omega(s)\|_{F}ds)^{2}\leq \int_{a}^{r}ds\int_{a}^{r}\|\omega(s)\|_{F}^{2}ds$$  $$\leq\int_{a}^{1}ds\int_{a}^{r}\|\omega(s)\|_{F}^{2}ds=(1-a)\int_{a}^{r}\|\omega(s)\|_{F}^{2}ds \leq (1-a)\int_{a}^{1}\|\omega(s)\|_{F}^{2}ds.$$ So we have obtained that $$\|K_{a}(\alpha)\|_{L^{2}(C^{*}(F),g_{c})}^{2}\leq (1-a)\int_{0}^{1}\int_{F}\int_{a}^{1}\|\omega(s)\|_{F}^{2}dsr^{c(f-2i+2)}dvol_{F}dr.$$

 Now consider the term $\int_{0}^{1}\int_{F}\int_{a}^{1}\|\omega(s)\|_{F}^{2}dsr^{c(f-2i+2)}dvol_{F}dr$ $$=\int_{0}^{1}\int_{F}\int_{a}^{1}\|\omega(s)\|_{F}^{2}(s^{c(f-2i+2)}+1-s^{c(f-2i+2)})dsr^{c(f-2i+2)}dvol_{F}dr.$$  We can bound the term $\int_{a}^{1}\|\omega(s)\|_{F}^{2}s^{c(f-2i+2)}ds$ in the following way $$\int_{a}^{1}\|\omega(s)\|_{F}^{2}s^{c(f-2i+2)}ds\leq \int_{0}^{1}\|\omega(s)\|_{F}^{2}s^{c(f-2i+2)}ds$$ and therefore $$\int_{F}\int_{a}^{1}\|\omega(s)\|_{F}^{2}s^{c(f-2i+2)}dsdvol_{F}\leq\int_{F}\int_{0}^{1}\|\omega(s)\|_{F}^{2}s^{c(f-2i+2)}dsdvol_{F}=\|\omega\|_{L^{2}(C^{*}(F),g_{c})}^{2}$$ while for the term $\int_{a}^{1}\|\omega(s)\|_{F}^{2}(1-s^{c(f-2i+2)})ds$ we can use the following observation:  there exist 
 $l>0$ such that $1-s^{c(f-2i+2)}\leq |1-s^{c(f-2i+2)}| \leq ls^{c(f-2i+2)}$ for any $s\in (\frac{1}{2},1]$. 
  Therefore: $$\int_{a}^{1}\|\omega(s)\|_{F}^{2}(1-s^{c(f-2i+2)})ds\leq \int_{a}^{1}\|\omega(s)\|_{F}^{2}|(1-s^{c(f-2i+2)})|ds\leq l\int_{a}^{1}\|\omega(s)\|_{F}^{2}s^{c(f-2i+2)}ds\leq$$ $$l\int_{0}^{1}\|\omega(s)\|_{F}^{2}s^{c(f-2i+2)}ds$$ 
  and similarly to the previous case we get $$\int_{F}\int_{a}^{1}\|\omega(s)\|_{F}^{2}(1-s^{c(f-2i+2)})dsdvol_{F}\leq l\|\omega\|_{L^{2}(C^{*}(F), g_{c})}^{2}$$ and the constant $l$ is independent of the choice of the form $\omega$ and of the choice of $a$. The fact that $i<\frac{f}{2}+\frac{1}{2c}+1$ implies that $\int_{0}^{1}r^{c(f-2i+2)}dr=\frac{1}{1+c(f-2i+2)}<\infty$ and so the following inequalities hold: $$\|K_{a}(\alpha)\|_{L^{2}(C^{*}(F),g_{c})}^{2}\leq (1-a)\int_{0}^{1}\int_{F}\int_{a}^{1}\|\omega(s)\|_{F}^{2}dsr^{c(f-2i+2)}dvol_{F}dr$$ $$\leq \int_{0}^{1}r^{c(f-2i+2)}dr(1-a)(1+l)\|\omega\|^{2}_{L^{2}(C^{*}(F,g_{c}))}\leq \frac{1}{2}\frac{1+l}{1+c(f-2i+2)}\|\alpha\|_{L^{2}(C^{*}(F),g_{c})}^{2}.$$ Therefore we can conclude that for $i<\frac{f}{2}+\frac{1}{2c}+1$ $$K_{a}:L^{2}\Omega^{i}(C^{*}(F),g_{c})\longrightarrow L^{2}\Omega^{i-1}(C^{*}(F),g_{c})$$ is a bounded operator uniformly in $a\in (\frac{1}{2},1).$ 
\end{proof} 

\begin{prop} Let $0<\rho<1$ and endow $(\rho,1)\times F$ with the metric $g_{c}$ restricted from $C^{*}(F)$. Let $\alpha=\phi+dr\wedge \omega \in L^{2}\Omega^{i}(C^{*}(F),g_{c}).$ If $i\geq \frac{f}{2}+\frac{1}{2c}$ then there exists a sequences $\epsilon_{s}\rightarrow 0$ such that $$\lim_{\epsilon_{s}\rightarrow 0}\|\phi(\epsilon_{s})\|_{L^{2}((\rho,1)\times F, g_{c})}^{2}=0$$
\label{frau}
\end{prop} 

\begin{proof} By the fact that $\alpha \in L^{2}\Omega^{i}(C^{*}(F),g_{c})$ follows that $\phi\in L^{2}\Omega^{i}(C^{*}(F),g_{c})$, so we know that $\int_{0}^{1}\int_{F}\|\phi(r)\|_{F}^{2}r^{c(f-2i)}dvol_{F}dr<\infty$. This means that $$\int_{F}\|\phi(r)\|_{F}^{2}r^{c(f-2i)}dvol_{F}\in L^{1}(0,1).$$ Thus by $\cite{C}$ lemma 1.2 there is a sequences $\epsilon_{s}\rightarrow 0$ for wich $$\left|\int_{F}\|\phi(\epsilon_{s})\|_{F}^{2}\epsilon_{s}^{c(f-2i)}dvol_{F}\right|<\frac{C}{\epsilon_{s}|ln(\epsilon_{s})|}$$ for some constant $C>0$. In this way we obtain $$\left|\int_{F}\|\phi(\epsilon_{s})\|_{F}^{2}dvol_{F}\right|<\frac{C\epsilon_{s}^{c(f-2i)-1}}{|ln(\epsilon_{s})|}.$$ Since $i\geq \frac{f}{2}+\frac{1}{2c}$ the right side tends to zero as $\epsilon_{s}\rightarrow 0.$ Thus we obtain: $$\|\phi(\epsilon_{s})\|_{L^{2}((\rho,1)\times F, g_{c})}^{2}=\int_{\rho}^{1}\int_{F}\|\phi(\epsilon_{s})\|_{F}^{2}\epsilon_{s}^{c(f-2i)}dvol_{F}dr$$ $$=\|\phi(\epsilon_{s})\|_{L^{2}( F,g)}^{2}\int_{\rho}^{1}r^{c(f-2i)}dr\longrightarrow 0$$ when $\epsilon_{s}\rightarrow 0.$ 
\end{proof} 

\begin{prop} If $i>\frac{f}{2}-\frac{1}{2c}+1$ and $\alpha= \phi+dr\wedge \omega\in L^{2}\Omega^{i}(C^{*}(F),g_{c})$, then $$K_{0}(\alpha)=\int_{0}^{r}\omega(s)ds\in L^{2}\Omega^{i-1}(C^{*}(F),g_{c})$$ and $K_{0}:L^{2}\Omega^{i}(C^{*}(F),g_{c})\longrightarrow L^{2}\Omega^{i-1}(C^{*}(F),g_{c})$ is a bounded operator.
\label{bau}
\end{prop} 

\begin{proof} By definition $$\|K_{0}(\alpha)\|_{L^{2}(C^{*}(F),g_{c})}^{2}=\int_{0}^{1}\int_{F}\|\int_{0}^{r}\omega(s)ds\|_{F}^{2}r^{c(f-2i+2)}dvol_{F}dr.$$
We consider the term $\|\int_{0}^{r}\omega(s)ds\|_{F}^{2}.$ Then:$$\|\int_{0}^{r}\omega(s)ds\|_{F}^{2}\leq (\int_{0}^{r}\|\omega(s)\|_{F}ds)^{2}=(\int_{0}^{r}s^{\frac{c}{2}(f-2i+2)}s^{\frac{c}{2}(2i-f-2)}\|\omega(s)\|_{F}ds)^{2}$$ and applying the Schwartz inequality we get that  $$\leq\int_{0}^{r}s^{c(2i-f-2)}ds\int_{0}^{r}s^{c(f-2i+2)}\|\omega(s)\|_{F}^{2}ds=\frac{r^{1+c(f-2i+2)}}{1+c(f-2i+2)}\int_{0}^{r}s^{c(f-2i+2)}\|\omega(s)\|_{F}^{2}ds.$$The last equality is a consequence of the fact that $i>\frac{f}{2}-\frac{1}{2c}+1$. Substituting the previous inequality in the definition of $\|K_{0}(\alpha)\|_{L^{2}(C^{*}(F),g_{c})}^{2}$ we get: $$\|K_{0}(\alpha)\|_{L^{2}(C^{*}(F),g_{c})}^{2}\leq\int_{0}^{1}\int_{F}\int_{0}^{1}s^{c(2i-f-2)}ds\int_{0}^{r}s^{c(f-2i+2)}\|\omega(s)\|_{F}^{2}dsdvol_{F}r^{c(f-2i+2)}dr$$ $$\leq\int_{0}^{1}\frac{r}{1+c(2i-f-2)}dr\int_{F}\int_{0}^{1}s^{c(f-2i+2)}\|\omega(s)\|_{F}^{2}dsdvol_{F}$$ $$=\frac{1}{2+2c(2i-f-2)}\|\omega\|_{L^{2}(C^{*}(F),g_{c})}^{2}\leq\frac{1}{2+2c(2i-f-2)}\|\alpha\|_{L^{2}(C^{*}(F),g_{c})}.$$ Thus $$K_{0}:L^{2}\Omega^{i}(C^{*}(F),g_{c})\longrightarrow L^{2}\Omega^{i-1}(C^{*}(F),g_{c})$$ is a bounded operator. 
\end{proof} 

\begin{prop} Let $$K_{\epsilon}(\alpha)=\int_{\epsilon}^{r}\omega(s)ds$$ and let $0<\rho<1$. If $i>\frac{f}{2}-\frac{1}{2c}+1$ then on $(\rho,1)\times F$ with the restricted metric $g_{c}$, $$K_{\epsilon}(\alpha)\longrightarrow K_{0}(\alpha)$$ in the $\|\ \|_{L^{2}((\rho,1)\times F,g_{c})}$ norm when $\epsilon\rightarrow 0$.
\label{miao}
\end{prop} 

\begin{proof} We have $$\|K_{\epsilon}(\alpha)-K_{0}(\alpha)\|=\int_{\rho}^{1}\int_{F}\|\int_{0}^{\epsilon}\omega(s)ds\|_{F}^{2}r^{c(f-2i+2)}dvol_{F}dr.$$Using the same techniques of the previous proof we obtain that the right hand side is at most $$\frac{\epsilon^{1+c(2i-f-2)}}{1+c(2i-f-2)}(\int_{\rho}^{1}r^{c(f-2i+2)}dr)\|\omega\|_{L^{2}(C^{*}(F),g_{c})}^{2}.$$ Since $i>\frac{f}{2}-\frac{1}{2c}+1$ the whole expression tends to $0$ as $\epsilon \rightarrow 0$.
\end{proof} 

\begin{prop} Let $(F,g)$ be an oriented riemannian manifold. Let $\phi\in \mathcal{D}(d_{max,i-1})\subset L^{2}\Omega^{i-1}(F,g)$, $\eta\in L^{2}\Omega^{i}(F,g)$ such that $d_{max,i-1}\phi=\eta$. Then for all $\rho\in (0,1)$ on $(\rho,1)\times F$ with the restricted metric $g_{c}$: 
\begin{enumerate}
\item $\pi^{*}\phi\in L^{2}\Omega^{i-1}((\rho,1)\times F)$
\item $\pi^{*}\eta\in L^{2}\Omega^{i}((\rho,1)\times F)$
\item  For all $\beta\in C_{0}^{\infty}\Omega^{i}((\rho,1)\times F)$ we have $$<\pi^{*}\phi,\delta_{i-1}\beta>_{L^{2}((\rho,1)\times F)}=<\pi^{*}\eta,\beta>_{L^{2}((\rho,1)\times F)}$$ that is on $(\rho,1)\times F$ with the restricted metric $g_{c}$ $$d_{max,i-1}\pi^{*}\phi=\pi^{*}\eta$$.
\end{enumerate}
\label{mr}
\end{prop} 

\begin{proof} $$\|\pi^{*}\phi\|_{L^{2}((\rho,1)\times F)}^{2}=\int_{\rho}^{1}r^{c(f-2i+2)}dr\int_{F}\|\phi\|_{F}^{2}dvol_{F}=\int_{\rho}^{1}r^{c(f-2i+2)}dr\|\phi\|_{L^{2}(F,g)}^{2}<\infty$$ so $\pi^{*}\phi\in L^{2}\Omega^{i-1}((1,\rho)\times F)$; $$\|\pi^{*}\eta\|_{L^{2}((\rho,1)\times F)}^{2}=\int_{\rho}^{1}r^{c(f-2i)}dr\int_{F}\|\eta\|_{F}^{2}dvol_{F}=\int_{\rho}^{1}r^{c(f-2i)}dr\|\eta\|_{L^{2}(F,g)}^{2}<\infty$$ so $\pi^{*}\eta\in L^{2}\Omega^{i}((1,\rho)\times F)$.\\By a  Cheeger's result, \cite{C}  pag 93,  $<\pi^{*}\phi,\delta_{i}\beta>_{L^{2}((\rho,1)\times F)}=<\pi^{*}\eta,\beta>_{L^{2}((\rho,1)\times F)}$ for all $\beta\in C_{0}^{\infty}\Omega^{i}((\rho,1)\times F)$ if and only if there is a sequence of smooth forms $\alpha_{j}\in L^{2}\Omega^{i-1}((\rho,1)\times F)$ such that $d_{i-1}\alpha_{j}\in L^{2}\Omega^{i}((\rho,1)\times F)$, $\|\pi^{*}\phi-\alpha_{j}\|_{L^{2}((\rho,1)\times F)}\rightarrow 0, \|\pi^{*}\eta-d_{i-1}\alpha_{j}\|_{L^{2}((\rho,1)\times F)}\rightarrow 0$ for $j\rightarrow \infty$. Using this  Cheeger's result , from the fact that $\phi\in Dom(d_{i-1,max})$, it follows that there is a sequences of smooth forms $\phi_{j}\in L^{2}\Omega^{i-1}(F,g)$ such that $d_{i-1}\phi_{j}\in L^{2}\Omega^{i}(F,g)$, $\|\phi-\phi_{j}\|_{L^{2}(F,g)}\rightarrow 0, \|\eta-d_{i-1}\phi_{j}\|_{L^{2}(F,g)}\rightarrow 0$ for $j\rightarrow \infty$. Now if we put $\alpha_{j}=\pi^{*}(\phi_{j})$ we obtain a sequence of smooth forms in $L^{2}\Omega^{i-1}((\rho,1)\times F)$ satisfying the assumptions of the same  Cheeger's result cited above. Indeed for each $j$ $$d_{i}\alpha_{j}\in L^{2}\Omega^{i}((\rho,1)\times F)$$ $$\|\alpha_{j}-\pi^{*}\phi\|_{{L^{2}((\rho,1)\times F)}}=\int_{\rho}^{1}r^{c(f-2i+2)}dr\int_{F}\|\phi-\alpha_{j}\|_{F}^{2}dvol_{F}\rightarrow 0$$ for $j\rightarrow \infty$ and similarly $$\|d\alpha_{j}-\pi^{*}\eta\|_{{L^{2}((\rho,1)\times F)}}\rightarrow 0$$ for $j\rightarrow \infty$. Therefore we can conclude that for all $\beta\in C_{0}^{\infty}\Omega^{i}((\rho,1)\times F)$  $$<\pi^{*}\phi,\delta_{i}\beta>_{L^{2}((\rho,1)\times F)}=<\pi^{*}\eta,\beta>_{L^{2}((\rho,1)\times F)}$$.
\end{proof} 

\begin{prop} Let $(F,g)$ be an oriented odd dimensional riemannian manifold such that\\ $d_{max,i-1}:\mathcal{D}(d_{max,i-1})\longrightarrow L^{2}\Omega^{i}(F,g)$ has closed range, where $i=\frac{f+1}{2}$ and $f=dimF$. Let $\alpha \in  L^2\Omega^{i}(C^{*}(F),g_{c})$ a smooth $i-$form such that $d_{i}\alpha \in L^2\Omega^{i+1}(C^{*}(F),g_{c})$. Then:
\begin{enumerate} 
\item  For almost all $b\in (0,1)$ there is an exact $i-$form $\eta_{b}\in \mathcal{D}(d_{max,i}) \subset L^2\Omega^{i}(F,g)$, $\eta_{b}=d_{max,i-1}\psi_{b},\ \psi_{b}\in \mathcal{D}(d_{max,i-1}) \subset L^{2}\Omega^{i-1}(F,g)$, such that for all $0<\rho<1$ on $(\rho,1)\times F$ with the restricted metric $g_{c}$ $$\|d_{i-1}(K_{b}\alpha)-(\alpha-K_{0}(d_{i}\alpha)-\pi^{*}(\eta_{b}))\|_{L^{2}((\rho,1)\times F)}=0$$
\item On $L^{2}\Omega^{i-1}(C^{*}(F),g_{c})$ we have $d_{max,i-1}(K_{b}\alpha+\pi^{*}(\psi_{b}))+K_{0}(d_{i}\alpha)=\alpha$
\end{enumerate}
\label{bei}
\end{prop} 

\begin{proof}  1) Let $\alpha=\phi+dr\wedge \omega.$ Consider $K_{\epsilon}(d_{i}\alpha)=\phi-\pi^{*}\phi(\epsilon)-\int_{\epsilon}^{r}d_{F}\omega ds$. Obviously for each $0<\rho<1$ $K_{\epsilon}(d_{i}\alpha)\in L^{2}\Omega^{i}((\rho,1)\times F)$ with the restricted metric $g_{c}$. From the fact that $\alpha$ is an $i-$ form and that $i+1=\frac{f+1}{2}+1>\frac{f}{2}+1-\frac{1}{2c}$ follows that we can use prop. \ref{miao} to conclude that $$K_{0}(d_{i}\alpha)\in L^{2}\Omega^{i}(C^{*}(F),g_{c})$$ and $$\|K_{\epsilon}(d_{i}\alpha)- K_{0}(d_{i}\alpha)\|_{L^{2}((\rho,1)\times F)}\rightarrow 0$$ for $\epsilon \rightarrow 0$. For the same reasons we can use prop. \ref{frau} to say that there is a sequence $\epsilon_{j}\rightarrow 0$ such that, on $(\rho,1)\times F$ with the restricted metric $g_{c}$,  $$\lim_{\epsilon_{j}\rightarrow 0}\|\pi^{*}\phi(\epsilon_{j})\|_{L^{2}((\rho,1)\times F, g_{c})}^{2}=0.$$ Therefore using these facts we can conclude that $$\lim_{\epsilon_{j}\rightarrow 0} \int_{\epsilon_{j}}^{r}d_{F}\omega ds\ \text{exists in}\ L^{2}\Omega^{i}((\rho,1)\times F)$$ and, if we call this limit $\gamma$, we have 
$$K_{0}(d_{i}(\alpha))=\phi-\gamma$$ in $\ L^{2}\Omega^{i}((\rho,1)\times F)$ with the restricted metric $g_{c}$.

 
 From this fact it follows that for almost all $b\in (0,1)$ $\int_{\epsilon_{j}}^{b}d_{F}\omega ds\rightarrow \gamma(b)$ in $L^{2}\Omega^{i}(F,g)$ for $\epsilon_{j} \rightarrow 0$. But $\int_{\epsilon_{j}}^{b}d_{F}\omega ds$ is a smooth form in $L^{2}\Omega^{i}(F,g)$; $\int_{\epsilon_{j}}^{b}\omega ds$ is a smooth form in $L^{2}\Omega^{i-1}(F,g)$ and $d_{i-1}(\int_{\epsilon_{j}}^{b}\omega ds)=\int_{\epsilon_{j}}^{b}d_{F}\omega ds.$ So we can conclude that  $\int_{\epsilon_{j}}^{b}d_{F}\omega ds=d_{max,i-1}(\int_{\epsilon_{j}}^{b}\omega ds)$ with $d_{max,i-1}:\mathcal{D}(d_{max,i-1})\rightarrow L^{2}\Omega^{i}(F,g)$. From this  it follows that      $\gamma(b)$ is in the closure of the image of $d_{max,i-1}:\mathcal{D}(d_{max,i-1})\rightarrow L^{2}\Omega^{i}(F,g)$ and so it follows from the assumptions that there is $\psi_{b}\in \mathcal{D}(d_{max,i-1})\subset L^{2}\Omega^{i-1}(F,g)$ such that $$d_{max,i-1}\psi_{b}=\gamma(b).$$ \\We choose one of these $b$ and $\epsilon$ such that $b>\epsilon$.\\ Now we consider $d_{i-1}(K_{b}(\alpha))=dr\wedge \omega+\int_{b}^{r}d_{F}\omega$. Adding $d_{i-1}(K_{b}(\alpha))$ and $K_{\epsilon}(d_{i}\alpha)$ we obtain $$d_{i-1}(K_{b}(\alpha))=\alpha-K_{\epsilon}(d_{i}\alpha)-\pi^{*}\phi(\epsilon)-\pi^{*}(\int_{\epsilon}^{b}d_{F}\omega ds)\in L^{2}\Omega^{i}((\rho,1)\times F))$$ with the restricted metric $g_{c}$ for all $\rho\in (0,1)$.\\  We analyze in detail the terms on the right of equality.  As noted above from the prop. \ref{frau} we know that there is a sequence $\epsilon_{j}\rightarrow 0$ such that $$\lim_{\epsilon_{j}\rightarrow 0}\|\pi^{*}\phi(\epsilon_{j})\|_{L^{2}((\rho,1)\times F, g_{c})}^{2}=0.$$ Similarly from the proposition \ref{miao} we know that $$\|K_{\epsilon_{j}}(d_{i}\alpha)- K_{0}(d_{i}\alpha)\|_{L^{2}((\rho,1)\times F)}\longrightarrow 0$$ for $\epsilon_{j} \rightarrow 0$. For the term $\pi^{*}(\int_{\epsilon_{j}}^{b}d_{F}\omega ds)$ we know, by the  observations made at the beginning of the proof and prop. \ref{mr}, that there is an $(i-1)-$form $\psi_{b}\in Dom(d_{max,i-1})\subset L^{2}\Omega^{i-1}(F,g)$ such that $$\|\pi^{*}(\int_{\epsilon_{j}}^{b}d_{F}\omega ds)- \pi^{*}(d_{max,i-1}(\psi_{b}))\|_{L^{2}((\rho,1)\times F)}\longrightarrow 0$$  for $\epsilon_{j} \rightarrow 0$. Summarizing, for all $\rho\in (0,1)$, we have on $(\rho,1)\times F$ with the restricted metric $g_{c}$ $$\lim_{\epsilon_{j}\rightarrow 0}\|\alpha-K_{\epsilon_{j}}(d_{i}\alpha)-\phi(\epsilon_{j})-\pi^{*}(\int_{\epsilon_{j}}^{b}d_{F}\omega ds)-(\alpha-K_{0}(d_{i}\alpha)-\pi^{*}(d_{i-1,max}(\psi_{b})))\|_{L^{2}((\rho,1)\times F)}=0.$$ Therefore, if we put $\eta_{b}=\gamma(b)$, by the fact that $$d_{i-1}(K_{b}(\alpha))=\alpha-K_{\epsilon_{j}}(d_{i}\alpha)-\pi^{*}\phi(\epsilon_{j})-\pi^{*}(\int_{\epsilon_{j}}^{b}d_{F}\omega ds)$$ for all $j$, we can conclude that $$\|d_{i-1}(K_{b}\alpha)-(\alpha-K_{0}(d_{i}\alpha)-\pi^{*}(\eta_{b}))\|_{L^{2}((\rho,1)\times F)}=0$$ \vspace{1 cm}

2) Before proving the statement we observe that from that fact that $i=\frac{f+1}{2}$ it follows that we can use prop \ref{qaz} to conclude that $K_{b}\alpha\in L^{2}\Omega^{i-1}(C^{*}(F),g_{c})$. Analogously we can use prop \ref{lop} to conclude that $\pi^{*}\psi_{b}\in L^{2}\Omega^{i-1}(C^{*}(F),g_{c})$. Let $\phi\in C^{\infty}_{0}\Omega^{i}(C^{*}(F))$. Then there is $\rho \in (0,1)$ such that $supp(\phi) \subset (\rho,1)\times F$.\\ We consider now: $$<K_{b}\alpha,\delta_{i-1}\phi>_{L^{2}(C^{*}(F),g_{c})}=<K_{b}\alpha,\delta_{i-1}\phi>_{L^{2}((\rho,1)\times F)}.$$ By the fact that $K_{b}(\alpha)$ is a smooth $(i-1)-$form such that $\|K_{b}(\alpha)\|_{L^{2}((1,\rho)\times F)}<\infty$,\\  $\|d_{i-1}(K_{b}\alpha)\|_{{L^{2}((1,\rho)\times F)}}<\infty$ and that $\phi$ is a smooth form with compact support it follows that: $$<K_{b}\alpha,\delta_{i-1}\phi>_{L^{2}((\rho,1)\times F)}=<d_{i-1}(K_{b}(\alpha),\phi>_{L^{2}((\rho,1)\times F)}=$$ $$=<\alpha-K_{0}(d_{i}\alpha)-\pi^{*}(\eta_{b}),\phi>_{L^{2}((\rho,1)\times F)}=$$ $$=<\alpha,\phi>_{L^{2}((\rho,1)\times F))}-<K_{0}(d_{i}\alpha),\phi>_{L^{2}((\rho,1)\times F)}-<\pi^{*}(\eta_{b}),\phi>_{L^{2}((\rho,1)\times F)}=$$ $$=<\alpha,\phi>_{L^{2}((\rho,1)\times F)}-<K_{0}(d_{i}\alpha),\phi>_{L^{2}((\rho,1)\times F)}-<\pi^{*}(\psi_{b}),\delta_{i-1}\phi>_{L^{2}((\rho,1)\times F)}=$$ $$=<\alpha,\phi>_{L^{2}(C^{*}(F),g_{c})}-<K_{0}(d_{i}\alpha),\phi>_{L^{2}(C^{*}(F),g_{c})}-<\pi^{*}(\psi_{b}),\delta_{i-1}\phi>_{L^{2}(C^{*}(F),g_{c})}.$$In particular the equality $<\pi^{*}(\psi_{b}),\delta_{i-1}\phi>_{L^{2}((\rho,1)\times F)}=<\pi^{*}(\eta_{b}),\phi>_{L^{2}((\rho,1)\times F)}$  follows from prop. \ref{mr}. We have obtained that for all $\phi\in C^{\infty}_{0}\Omega^{i}(C^{*}(F))$ $$<K_{b}\alpha+\pi^{*}\psi_{b},\delta_{i-1}\phi>_{L^{2}(C^{*}(F),g_{c})}=<\alpha-K_{0}(d_{i}\alpha),\phi>_{L^{2}(C^{*}(F),g_{c})}.$$ So we can conclude that $$d_{max,i-1}(K_{b}\alpha+\pi^{*}(\psi_{b}))+K_{0}(d_{i}\alpha)=\alpha.$$
\end{proof}

\section{$L^{2}$ cohomology of a cone over a riemannian manifold} 

In this section we continue to use the notations of the previous section.

\begin{teo} Let $(F,g)$ be an oriented riemannian manifold. Then for the riemannian manifold $(C^{*}(F),g_{c})$, with $g_{c}$ as in \eqref{fcc} the following isomorphism holds: 

\begin{equation}
H^{i}_{2,max}(C^{*}(F),g_{c}) = \left\{  
\begin{array}{ll}
H_{2,max}^{i}(F,g) & i<\frac{f}{2}+\frac{1}{2c}\\
0& i>\frac{f}{2}+1-\frac{1}{2c}
\end{array}
\right.
\end{equation} 
\label{lll}
\end{teo}

\begin{proof} For the first part of the proof we use the complex $(\Omega^{*}_{2}(C^{*}(F),g_{c}),d_{*})$ of  prop. \ref{zaq}. Let $\alpha\in \Omega^{i}_{2}(C^{*}(F),g_{c})$, $\alpha=\phi+dr\wedge \omega$, $i=0,...,f+1$. Let $a\in (\frac{1}{2},1)$. Consider the following map 
\begin{equation}
v_{a}:\Omega^{i}_{2}(C^{*}(F),g_{c})\rightarrow \Omega^{i}_{2}(F,g),\ v_{a}(\alpha)=\phi(a).
\label{mkmk}
\end{equation}
By prop.\ref{yui} $v_{a}(\alpha)\in L^{2}\Omega^{i}(F,g)$. Furthermore this map satisfies $v_{a}\circ d_{i}=d_{i}\circ v_{a}$ where on the left of the equality $d_{i}$ is the $i-$th differential of the complex $(\Omega_{2}^{*}(C^{*}(F),g_{c}),d_{*})$ while on the right of the equality the operator $d_{i}$ is the $i-$th differential of the complex $(\Omega_{2}^{*}(F,g),d_{*})$.
 Therefore $v_{a}$ is a morphism between the complex $(\Omega_{2}^{*}(C^{*}(F),g_{c}),d_{*})$ and the complex $(\Omega_{2}^{*}(F,g),d_{*})$ so it induces a map between the cohomology groups 
 \begin{equation}
 v_{a}^{*}:H_{2}^{i}(C^{*}(F),g_{c})\rightarrow H_{2}^{i}(F,g)
 \label{mkm}
 \end{equation}
 where $H_{2}^{i}(F,g)$ is the $i-th$ cohomology group of the complex $(\Omega_{2}^{*}(F,g),d_{*}).$\\Now in the case $i<\frac{f}{2}+\frac{1}{2c}$, by proposition  \ref{qaz}, we know that  $K_{a}(\alpha)$ and $K_{a}(d_{i}\alpha)$ are two smooth form such that\\ $\|K_{a}(d_{i}\alpha)\|_{L^{2}(C^{*}(F),g_{c})}<\infty$ and $\|K_{a}\alpha\|_{L^{2}(C^{*}(F),g_{c})}<\infty$. If we add the two following terms, $d_{i-1}(K_{a}(\alpha))$ and $K_{a}(d_{i}(\alpha))$ we obtain: 
\begin{equation}
d_{i-1}(K_{a}\alpha)+K_{a}(d_{i}(\alpha))=dr\wedge \omega(s)ds+\int_{a}^{r}d_{F}(s)ds\omega+\phi-\phi(a)-\int_{a}^{r}d_{F}(s)ds\omega=\alpha-\pi^{*}(v_{a}(\alpha)).
\label{oooo}
\end{equation}
 So we have obtained that $\|d_{i-1}(K_{a}\alpha)\|_{L^{2}(C^{*}(F),g_{c})}<\infty$ and from this and \eqref{oooo} it follows that 
  $$(\pi^{*})^*\circ v_{a}^*:H_{2}^{i}(C^{*}(F),g_{c})\rightarrow H_{2}^{i}(C^{*}(F),g_{c})$$ is 
 an isomorphism for $i<\frac{f}{2}+\frac{1}{2c}$. Now from this fact it follows that for the same $i$: $$v_{a}^{*}:H_{2}^{i}(C^{*}(F),g_{c})\rightarrow H_{2}^{i}(F,g)$$ is injective and that $$(\pi^{*})^{*}: H_{2}^{i}(F,g)\rightarrow H_{2}^{i}(C^{*}(F),g_{c})$$ is surjective. But from prop. \ref{lop} we know that $v_{a}^{*}:H_{2}^{i}(C^{*}(F),g_{c})\rightarrow H_{2}^{i}(F,g)\ is\ surjective$. So for $i<\frac{f}{2}+\frac{1}{2c}$ $H_{2}^{i}(C^{*}(F),g_{c})$ and $H_{2}^{i}(F,g)$  are isomorphic and therefore by proposition \ref{zaq} for the same $i$ we have $$H_{2,max}^{i}(C^{*}(F), g_{c})\cong H_{2,max}^{i}(F, g).$$

Now we start the second part of the proof. We know that for each $i$ every cohomology class $[\alpha]\in H_{2,max}^{i}(C^{*}(F))$ has a smooth representative. So let $\alpha\in L^{2}\Omega^{i}(C^{*}(F),g_{c})$, $i>\frac{f}{2}+1-\frac{1}{2c}$, a smooth form such that $d_{i}\alpha=0$. Observe that from the fact that $\alpha$ is closed follows that $\phi^{'}=d_{F}\omega$ and therefore, given $\epsilon\in (0,1)$ we have $d_{i-1}(K_{\epsilon}\alpha)=d_{i-1}(\int_{\epsilon}^{r}\omega(s)ds)=dr\wedge \omega+\int_{\epsilon}^{r}d_{F}\omega(s)ds=dr\wedge \omega+\int_{\epsilon}^{r}\phi^{'}(s)ds=dr\wedge \omega+\phi-\phi(\epsilon)=\alpha-\phi(\epsilon)$. Consider $K_{0}(\alpha)$; by proposition \ref{bau} we know that $K_{0}(\alpha)\in L^{2}\Omega^{i}(C^{*}(F),g_{c})$. We want to show that $d_{max,i-1}(K_{0}(\alpha))=\alpha$.\\Let $\beta\in C^{\infty}_{0}\Omega^{i}(C^{*}(F))$. Then there is $\rho>0$ such that $supp(\beta)\subset (\rho,1)\times F$. Therefore: $$<K_{0}\alpha,\delta_{i-1}\beta>_{L^{2}(C^{*}(F))}=<K_{0}\alpha,\delta_{i-1}\beta>_{L^{2}((\rho,1)\times F)}= (by\ prop\ 
\ref{miao})$$ $$=\lim_{\epsilon\rightarrow 0}<K_{\epsilon}\alpha,\delta_{i-1}\beta>_{L^{2}((\rho,1)\times F)}.$$By the fact that $K_{\epsilon}(\alpha)$ is a smooth form such that $\|K_{\epsilon}(\alpha)\|_{L^{2}((1,\rho)\times F)}<\infty$,\\ $\|d_{i-1}(K_{\epsilon}\alpha)\|_{{L^{2}((1,\rho)\times F)}}<\infty$ and that $\phi$ is a smooth form with compact support it follows that:$$\lim_{\epsilon\rightarrow 0}<K_{\epsilon}\alpha,\delta_{i-1}\beta>_{L^{2}((\rho,1)\times F)}=\lim_{\epsilon\rightarrow 0}<d_{i-1}(K_{\epsilon}\alpha),\beta>_{L^{2}((\rho,1)\times F)}=$$ $$=\lim_{\epsilon\rightarrow 0}<\alpha-\phi(\epsilon),\beta>_{L^{2}((\rho,1)\times F)}=<\alpha,\beta>_{L^{2}((\rho,1)\times F)}-\lim_{\epsilon\rightarrow 0}<\phi(\epsilon),\beta>_{L^{2}((\rho,1)\times F)}.$$ In particular the limit $$\lim_{\epsilon\rightarrow 0}<\phi(\epsilon),\beta>_{L^{2}((\rho,1)\times F)}$$ exist. But from prop. \ref{frau} we know that there is a sequence $\epsilon_{j}\rightarrow 0$ such that $$\lim_{\epsilon_{j}\rightarrow 0}<\phi(\epsilon_{j}),\beta>_{L^{2}((\rho,1)\times F)}=0.$$ Therefore $$<K_{0}\alpha,\delta_{i-1}\beta>_{L^{2}((\rho,1)\times F)}= <\alpha,\delta_{i-1}\beta>_{L^{2}((\rho,1)\times F)}=<\alpha,\delta_{i-1}\beta>_{L^{2}(C^{*}(F),g_{c})}.$$ Thus we can conclude that $d_{max,i-1}(K_{0}(\alpha))=0$ and hence that $H_{2,max}^{i}(C^{*}(F),g_{c})=0$ for $i>\frac{f}{2}+1-\frac{1}{2c}$.
\end{proof} 

\begin{cor} Suppose that  one of three following hypotheses applies:
\begin{enumerate}
\item $0<c<1$.
\item $c\geq 1$ and $f=dimF$ is even.
\item $c\geq 1$, $f$ is odd and $d_{max,i-1}:\mathcal{D}(d_{max,i-1})\rightarrow L^{2}\Omega^{i}(F,g)$ has close range where $i=\frac{f+1}{2}$. (By prop \ref{qwe} this happen for example when $H^{i}_{2,max}(F,g)$ is finite dimensional.)
\end{enumerate}

Then for the riemannian manifold $(C^{*}(F),g_{c})$ the following isomorphism holds: 
\begin{equation}
H^{i}_{2,max}(C^{*}(F),g_{c}) = \left\{  
\begin{array}{ll}
H_{2,max}^{i}(F,g) & i<\frac{f}{2}+\frac{1}{2c}\\
0& i\geq\frac{f}{2}+\frac{1}{2c}
\end{array}
\right.
\label{pol}
\end{equation} 
\label{edo}
\end{cor}

\begin{proof} If $0<1<c$ then $\frac{f}{2}+\frac{1}{2c}> \frac{f}{2}+1-\frac{1}{2c}$.\\If $c\geq 1$ and $f$ is even then $i>\frac{f}{2}+1-\frac{1}{2c}$ if and only if $i\geq\frac{f}{2}+\frac{1}{2c}$.\\ Finally if $c\geq 1$, $f$ is odd and $d_{max,i-1}:Dom(d_{max,i-1})\rightarrow L^{2}\Omega^{i}(F,g)$ has close range then the thesis immediately follows from prop. \ref{bei}.

\end{proof}

\begin{rem} Now we make a simple remark; theorem \ref{lll} also holds in the following two cases:
\begin{enumerate}
\item If we replace $C(F)$ with $C_{\epsilon}(F)$ where $C_{\epsilon}(F)=F\times [0,\epsilon)/F\times \{0\}$ and where $\epsilon$ is any real positive number. In this case we have only to modify prop. \ref{yui} and prop. \ref{qaz} choosing $a\in (\gamma, \epsilon)$ where $\gamma$ is a fixed and positive real number strictly smaller than $\epsilon$.  Furthermore if $\epsilon<\delta$  $$i^{*}:(L^{2}\Omega^{*}(C^{*}_{\delta}(F),g_{c}),d_{max,*})\rightarrow (L^{2}\Omega^{*}(C^{*}_{\epsilon}(F),g_{c}),d_{max,*})$$ where $i^{*}$ is the morphism of complexes induced by the inclusion $i:C_{\epsilon}(F)\rightarrow C_{\delta}(F)$, induces an isomorphism between the cohomology groups $H^{i}_{2,max}(C_{\epsilon}^{*}(F),g_{c})$ and\\ $H^{i}_{2,max}(C_{\delta}^{*}(F),g_{c})$ for each $i<\frac{f}{2}+\frac{1}{2c}$ or $i>\frac{f}{2}+1-\frac{1}{2c}$. This last assertion is easy to see. When $i>\frac{f}{2}+1-\frac{1}{2c}$ it is obvious because the cohomology groups are both null; when $i<\frac{f}{2}+\frac{1}{2c}$ it follows by the fact that given $a\in (\gamma,\epsilon)$ and given $v_{a}$, which is the evaluation map defined like  in \eqref{mkmk}, we have $v_{a}=v_{a}\circ i^{*}$ where at the left of the equality $v_{a}$ is between $\Omega_{2}^{i}(C_{\delta}^{*}(F),g_{c})$ and $\Omega_{2}^{i}(F,g)$  and at the right of the equality it is between $\Omega_{2}^{i}(C_{\epsilon}^{*}(F),g_{c})$ and $\Omega_{2}^{i}(F,g)$ . Finally if the hypotheses of corollary \ref{edo} holds then the same corollary holds for $C_{\epsilon}^{*}(F)$ and in this case $i^{*}$ induces an isomorphism between $H^{i}_{2,max}(C_{\epsilon}^{*}(F),g_{c})$ and $H^{i}_{2,max}(C_{\delta}^{*}(F),g_{c})$  for all $i$.
\item When $(F,g)$ is a disconnected riemannian manifold made of a finite number of connected components all having  the same dimension,  that is $(F,g)=\bigcup_{j\in J}(F_{j},g_{j})$, $dimF_{i}=dimF_{j}$ for each $i,j\in J$ and $J$ is finite. Indeed in this case: 
\begin{equation}
H^{i}_{2,max}(C^{*}(F),g_{c})=H^{i}_{2,max}(C^{*}(\bigcup_{j\in J}F_{j}),g_{c})=\bigoplus_{j\in J}H^{i}_{2,max}(C^{*}(F_{j}),g_{c,j})
\end{equation}
\begin{equation}
=\bigoplus_{j\in J} \left\{  
\begin{array}{ll}
H_{2,max}^{i}(F_{j},g_{j}) & i<\frac{f}{2}+\frac{1}{2c}\\
0& i>\frac{f}{2}+1-\frac{1}{2c}
\end{array}
\right.
= \left\{  
\begin{array}{ll}
H_{2,max}^{i}(F,g) & i<\frac{f}{2}+\frac{1}{2c}\\
0& i>\frac{f}{2}+1-\frac{1}{2c}
\end{array}
\right.
\end{equation}
Obviously if each $(F_{j},g_{j})$ satisfies the assumptions of corollary \ref{edo} then also corollary \ref{edo} holds for $(C^{*}(F),g_{c})$. This situation could happen in theorem \ref{ris} of the next section. In that case the manifold $F$ will be the regular part of a link and it could happen that it is disconnected.
\end{enumerate}
\label{sdds}
\end{rem}

We conclude the section  recalling a result from \cite{C} that we will use in the proof of theorem \ref{ris}.

\begin{prop} Let $(M,g)$ be a Riemannian manifold. Then for the riemannian manifold $((0,1)\times M,dr\otimes dr+g)$ the following isomorphism holds:
\begin{equation}
 H^{i}_{2,max}((0,1)\times M,dr\otimes dr+g)\cong H^{i}_{2,max}(M,g)\ for\ all\ i=0,...,dimM+1
\end{equation} 
\label{ioi}
\end{prop}

\begin{proof}
See \cite{C} pag 115.
\end{proof}



\section{ $L^{2}$ Hodge and de Rham theorems}

Before starting the section we make a remark about the notation. Given an open subset $U\subset X$ with $\mathcal{D}(U,d_{max/min,i})$ we mean the domain of $d_{max/min,i}$ in $L^{2}\Omega^{i}(reg(U),g|_{reg(U)})$ Given a complex of sheaves $(\mathcal{L}^{*}, d_{*})$ over $X$ and an open subset $U$ of $X$ with the symbol $H^{i}(\mathcal{L}^{*}(U), d_{*})$ we mean the $i-$th cohomology group of the complex $$...\stackrel{d_{i-2}}{\rightarrow}\mathcal{L}^{i-1}(U)\stackrel{d_{i-1}}{\rightarrow}\mathcal{L}^{i}(U)\stackrel{d_{i}}{\rightarrow}\mathcal{L}^{i+i}(U)\stackrel{d_{i+1}}{\rightarrow}...$$ Finally with $\mathbb{H}^{i}(\mathcal{L}^{*}, d_{*})$ we mean the $i-$th cohomology sheaf associated to the complex $(\mathcal{L}^{*}, d_{*}).$

\begin{teo} Let $X$ be a compact and oriented smoothly stratified pseudomanifold of dimension $n$ with a  Thom-Mather stratification  $\mathfrak{X}$. Let $g$ be a quasi  edge metric with  weights on $reg(X)$, see definition \ref{zedge}. Let $\mathcal{R}_{0}$ be the stratified coefficient system made of  the pair of coefficient systems given by $(X-X_{n-1})\times \mathbb{R}$ over $X-X_{n-1}$ where the fibers $\mathbb{R}$ have the discrete topology  and the constant $0$ system on $X_{n-1}$. Let $p_{g}$ be the general perversity associated to the metric $g$, see definition \ref{pim}. Then, for all $i=0,...,n$,  the following isomorphisms holds:
\begin{equation}
I^{q_{g}}H^{i}(X,\mathcal{R}_{0})\cong H_{2,max}^{i}(reg(X),g)\cong \mathcal{H}_{abs}^{i}(reg(X),g)
\label{max}
\end{equation}
\begin{equation}
I^{p_{g}}H^{i}(X, \mathcal{R}_{0})\cong H_{2,min}^{i}(reg(X),g)\cong \mathcal{H}_{rel}^{i}(reg(X),g)
\label{min}
\end{equation} 
where $q_{g}$ is the complementary perversity of $p_{g}$, that is, $q_{g}=t-p_{g}$ and $t$ is the usual top perversity. In particular, for all $i=0,...,n$ the groups $$H_{2,max}^{i}(reg(X),g),\ H_{2,min}^{i}(reg(X),g),\ \mathcal{H}_{abs}^{i}(reg(X),g),\ \mathcal{H}_{rel}^{i}(reg(X),g)$$ are all finite dimensional.
\label{ris}
\end{teo}

\begin{teo} Let $X$ be as in the previous theorem. Let $p$ a general perversity in the sense of Friedman on $X$. If $p$ satisfies the following conditions:
\begin{equation}
 \left\{  
\begin{array}{ll}
p\geq \overline{m}\\
p(Y)=0 & if\ cod(Y)=1
\end{array}
\right.
\end{equation}
then there exists $g$, a quasi edge edge metric with weights on $reg(X)$, such that
\begin{equation}
I^{p}H^{i}(X, \mathcal{R}_{0})\cong H_{2,min}^{i}(reg(X),g)\cong \mathcal{H}_{rel}^{i}(reg(X),g).
\label{minn}
\end{equation} 
Conversely if $p$ satisfies:
\begin{equation}
 \left\{  
\begin{array}{ll}
p\leq \underline{m} \\
p(Y)=-1 & if\ cod(Y)=1
\end{array}
\right.
\end{equation}
then, also in this case, there exists a quasi edge metric with weights $h$ on $reg(X)$ such that 
\begin{equation}
I^{p}H^{i}(X,\mathcal{R}_{0})\cong H_{2,max}^{i}(reg(X),h)\cong \mathcal{H}_{abs}^{i}(reg(X),h).
\label{maxx}
\end{equation}
\label{new}
\end{teo}

Before proving these theorems we need some preliminary results.

\begin{prop} Let $X$ be  an oriented smoothly stratified pseudomanifold of dimension $n$ with a Thom-Mather stratification  and let $g$ a riemannian metric  on $reg(X)$. Consider, for every $i=0,...,n$, the following presheaf:
\begin{equation}
U\longmapsto \mathcal{D}(U,d_{max,i})= \left\{
\begin{array}{ll}
\mathcal{D}(U,d_{max,i})& U\cap X_{n-1}=\emptyset\\
\mathcal{D}(U-(U\cap X_{n-1}),d_{max,i})& U\cap X_{n-1}\neq \emptyset
\end{array}
\right.
\end{equation}
or
\begin{equation}
U\longmapsto \left\{
\begin{array}{ll}
\omega \in \Omega^{i}_{2}(U,g|_{U}) & U\cap X_{n-1}=\emptyset\\
\omega \in \Omega^{i}_{2}(reg(U),g|_{reg(U)}) & U\cap X_{n-1}\neq \emptyset
\end{array}
\right.
\end{equation}
Let $\mathcal{L}^{i}_{2,max}$ and $\mathcal{L}^{i}_{2}$ be the sheaves associated to the previous presheaves;  then for these sheaves we have the following explicit descriptions:
\begin{enumerate}
\item let U an open subset of $X$ then: $\mathcal{L}^{i}_{2,max}(U)\cong\{\omega\in L^{2}_{Loc}\Omega^{i}(reg(U),g|_{reg(U)}): \forall\ p\in U\ \exists\ V\ open\ neighbourhood\ of\ p\ in\ U\ such\ that$ $\omega|_{reg(V)}\in \mathcal{D}(reg(V), d_{max,i})\}.$ 
\item $\mathcal{L}^{i}_{2}(U)\cong\{\omega\in \Omega^{i}(reg(U),g|_{reg(U)}): \forall\ p\in U\ \exists\ V\ open\ neighbourhood\ of\ p\ in\ U\ such\ that$ $\omega|_{reg(V)}\in \Omega^{i}_{2}(reg(V),g|_{reg(V)})\}.$  
\item If $X$ is compact $\mathcal{L}^{i}_{2,max}(X)=\mathcal{D}(reg(X),d_{max,i})$. 
\item $\mathcal{L}^{i}_{2}(X)=\{\omega\in \Omega^{i}(reg(X)): \omega\in L^{2}\Omega^{i}(reg(X),g),$ $\ d_{i}\omega\in L^{2}\Omega^{i}(reg(X),g)\}.$ 
\item The complexes $\mathcal{L}^{i}_{2,max}$ and $\mathcal{L}^{i}_{2}$ are quasi isomorphic.
\end{enumerate}
\label{ddd}
\end{prop}

\begin{proof}
The first and the second statement follow from the fact that the sheaves $\mathcal{L}^{i}_{2,max},\ \mathcal{L}^{i}_{2}$ and the respective sheaves at the right of $\cong$ have isomorphic stalks. The third and fourth statement are an immediate consequences of the compactness of $X$. The fifth statement follows immediately from proposition \ref{zaq}. 
\end{proof}

\begin{prop} 
Let $X$ be an oriented smoothly stratified pseudomanifold with a Thom-Mather stratification of dimension $n$ such that for each stratum $Y$ the link $L_{Y}$ is compact and $g$ a quasi rigid iterated  edge metric with weights on $reg(X)$. 
Then, for each $i=0,...,n$, $\mathcal{L}^{i}_{2,max}$ and $\mathcal{L}^{i}_{2}$ are  fine sheaves.
\label{cfs}
\end{prop}

\begin{proof} From the description of the sheaves $\mathcal{L}^{i}_{2,max},\ \mathcal{L}^{i}_{2}$ given in prop. \ref{ddd} it follows that  in order to prove this proposition  it is sufficient to show that on $X$, given an open cover $\mathcal{U}_{A}=\{U_{\alpha}\}_{\alpha\in A}$, there is a bounded partition of unity with bounded differential subordinate to $\mathcal{U}_{A}$, that is a family of functions $\lambda_{\alpha}:X\rightarrow [0,1], \alpha\in A$ such that
\begin{enumerate}
\item Each $\lambda_{\alpha}$ is continuous and $\lambda_{\alpha}|_{reg(X)}$ is smooth.
\item  $supp(\lambda_{\alpha})\subset U_{\alpha}$ for some $\alpha\in A$.
\item $\{supp(\lambda_{\alpha})\}_{\alpha\in A}$ is a locally finite cover of $X$.
\item For each $x\in X$ $\sum_{\alpha\in A}\lambda_{\alpha}(x)=1$.
\item There are constants $C_{\alpha}>0$ such that each $\lambda_{\alpha}$ satisfies $\|d(\lambda_{\alpha}|_{reg(X)})\|_{L^{2}(reg(X),g)}\leq C_{\alpha}$.
\label{ygn}
\end{enumerate}

The proof is given by induction on the depth of $X$. If $depth(X)=0$ the statement is immediate because in this case $X$ is a differentiable  manifold.
Suppose now that the statement is true if $depth(X)\leq k-1$ and that $depth(X)=k$.
Let $\mathcal{U}_{J}=\{U_{j}\}_{j\in J}$ be a locally finite refinement of $\mathcal{U}_{A}$ such that for each $U_{J}$  there is a diffeomorphism $\phi_{j}:U_{j}\rightarrow \mathbb{R}^n$ if $U_{j}\cap X_{n-1}=\emptyset$ or, in the case $U_{j}\cap X_{n-1}\neq \emptyset$, an isomorphism $\phi_{j}:U_{j}\rightarrow W_{j}\subset \mathbb{R}^{k}\times C(L_{j})$ between $U_{j}$ and an open subset, $W_{j}$, of the product $\mathbb{R}^{k}\times C(L_{j})$ for some $k<n$ and stratified space $L_{j}$.\\ Let $\mathcal{V}_{J}=\{V_{j}\}_{j \in J}$  a shrinking of $\mathcal{U}_{J}$; this means that $\mathcal{V}_{J}$ is a refinement of $\mathcal{U}_{J}$ such that if $V_{j}\subset U_{j}$ then $\overline{V_{j}}\subset U_{j}$. 
  Now let $V_{j}\in \mathcal{V}_{J}$, $U_{j}\in \mathcal{U}_{J}$  such that $V_{j}\subset U_{j}$ and $U_{j}\cap X_{n-1}=\emptyset$. Let $\psi_{j}:\mathbb{R}^{n}\rightarrow [0,1]$ be a smooth function such that $\psi_{j}|_{\overline{\phi_{j}(V_{j})}}=1$ and $supp(\psi_{j})\subset \phi_{j}(U_{j})$. Define $\lambda_{j}:X\rightarrow [0,1], \lambda_{j}:=\psi_{j}\circ \phi_{j}$.  Now let $V_{j}\in \mathcal{V}_{J}$, $U_{j}\in \mathcal{U}_{J}$ such that $V_{j}\subset U_{j}$ and $U_{j}\cap X_{n-1} \neq \emptyset$. We can take two functions $\eta:\mathbb{R}^{k}\rightarrow [0,1]$, $\xi:[0,1)\rightarrow [0,1]$ and, using the inductive hypothesis and the fact that $L_{Y}$ is compact, a third function $\tau_{j}:L_{j}\rightarrow [0,1]$ smooth on $reg(L_{j})$ and with bounded differential such that 
  $\psi_{j}:=\eta_{j}\xi_{j}\tau_{j}$ is a  a continuous function on $\mathbb{R}^{k}\times C(L_{j})\rightarrow [0,1]$ smooth on the regular part and with bounded differential such that $\psi_{j}|_{\overline{\phi_{j}(V_{j})}}=1$ and $supp(\psi_{j})\subset \phi_{j}(U_{j})$. Also in this case define $\lambda_{j}:X\rightarrow [0,1],\ \lambda_{j}:=\psi_{j}\circ \phi_{j}$. Finally define 
\begin{equation}
\mu_{j}:X\rightarrow [0,1], \mu_{j}=\frac{\lambda_{j}}{\sum_{j\in J}\lambda_{j}}
\end{equation}
$\{\mu_{j}\}_{\in J}$ is a partition of unity with bounded differential subordinated to the cover $\mathcal{U}_{J}$ and therefore from this follows immediately that there exist a partition of unity with bounded differential subordinated to the cover $\mathcal{U}_{A}$.
Now the statement of the proposition is an immediate consequence.
\end{proof}

Now we state the last proposition that we will use in the proof of theorem \ref{ris}.

\begin{prop} Let $L$ be a compact smoothly stratified pseudomanifold with a Thom-Mather stratification and let $g_{L}$ be a riemannian metric  on $reg(L)$. Let $C(L)$ be the cone over $L$ and on $reg(C(L))$ consider the metric $dr\otimes dr+r^{2c}g_{L}$. 
Finally consider on $C(L)$ the complex of sheaves $(\mathcal{L}^{*}_{2,max}, d_{max,*})$ associated to the metric $dr\otimes dr+r^{2c}g_{L}$. Then the canonical inclusion $$i_{v}:C(L)-\{v\}\longrightarrow C(L),$$ where $v$ is the vertex of the cone, induces a quasi-isomorphism between   the complexes $$(\mathcal{L}^{*}_{2,max},d_{max,*})\ and\ (i_{v*}i_{v}^{*}\mathcal{L}^{*}_{2,max},d_{max,*})$$  for  $i\leq [[\frac{dimL}{2}+\frac{1}{2c}]]$.
\label{bea}
\end{prop}

\begin{proof}  We start the proof showing that the complexes $(\mathcal{L}^{*}_{2,max},d_{max,*})$ and $(i_{v*}i_{v}^{*}\mathcal{L}^{*}_{2,max},d_{max,*})$ are quasi isomorphic for  $i\leq [[\frac{dimL}{2}+\frac{1}{2c}]]$. This is equivalent to show that for each $x\in C(L)$ $$(\mathbb{H}^{i}(\mathcal{L}^{*}_{2,max},d_{max,*}))_{x}\cong (\mathbb{H}^{i}(i_{v*}i_{v}^{*}\mathcal{L}^{*}_{2,max},d_{max,*}))_{x}$$ where each term in the previous isomorphism is the stalk at the point $x$ of the $i-$th cohomology sheaf associated to $(\mathcal{L}^{*}_{2,max},d_{max,*})$ and $(i_{v*}i_{v}^{*}\mathcal{L}^{*}_{2,max},d_{max,*})$ respectively.  For every $i=0,...,dimL+1$ the sheaf $i_{v*}i_{v}^{*}\mathcal{L}^{i}_{2,max}$ is isomorphic to the following sheaf; let $U\subset C(L)$ an open subset then:

  $$i_{v*}i_{v}^{*}\mathcal{L}^{i}_{2,max}(U)\cong\{\omega\in L^{2}_{Loc}\Omega^{i}(reg(U),dr\otimes dr+r^{2c}g_{L}|_{reg(U)}): \forall\ p\in U-\{v\}\     \exists\ V\ open$$ $$ neighbourhood\ of\ p\ in\ U\ such\ that\ \omega|_{reg(V)}\in \mathcal{D}(reg(V), d_{max,i})\}.$$

From this fact and prop. \ref{ddd} it follows that for every $x\in C(L)-\{v\}$ 
\begin{equation}
(\mathbb{H}^{i}(\mathcal{L}^{*}_{2,max},d_{max,*}))_{x}\cong (\mathbb{H}^{i}(i_{v*}i_{v}^{*}\mathcal{L}^{*}_{2,max},d_{max,*}))_{x}.
\label{frutta}
\end{equation} 
Now by theorem  \ref{lll} and remark \ref{sdds} we know that for $i\leq [[\frac{dimL}{2}+\frac{1}{2c}]]$ $$(\mathbb{H}^{i}(\mathcal{L}^{*}_{2,max},d_{max,*}))_{v}\cong H^{i}(\mathcal{L}^{*}_{2,max}(C(L)),d_{max,*}) \cong H^{i}_{2,max}(reg(L),g_{L}).$$ Using the same techniques it is easy to show that for each $i$ $$(\mathbb{H}^{i}(i_{v*}i_{v}^{*}\mathcal{L}^{*}_{2,max},d_{max,*}))_{v}\cong H^{i}(i_{v*}i_{v}^{*}\mathcal{L}^{*}_{2,max}(C(L)),d_{max,*}).$$ Therefore we have to show that for $i\leq [[\frac{dimL}{2}+\frac{1}{2c}]]$ $$H^{i}(i_{v*}i_{v}^{*}\mathcal{L}^{*}_{2,max}(C(L)),d_{max,*})\cong H^{i}_{2,max}(reg(L),d_{max,*}).$$
 On the whole cone  $C(L)$ the main difference between the complexes  $(\mathcal{L}^{*}_{2,max},d_{max,*})$ and\\ $(i_{v*}i_{v}^{*}\mathcal{L}^{*}_{2,max},d_{max,*})$ is that for  each $\omega\in \mathcal{L}^{i}_{2,max}(L)$, by prop. \ref{lop}, $$\pi^{*}\omega\in \mathcal{L}^{i}_{2,max}(C(L))\ \text{if}\ \text{and}\ \text{only}\ \text{if}\ i<\frac{dimL}{2}+\frac{1}{2c}.$$
Instead $$\pi^{*}\omega\in i_{v*}i_{v}^{*}\mathcal{L}^{i}_{2,max}(C(L))\ \text{for}\ \text{every}\ i=0,...,dimL.$$ Therefore by the proof of the first part of  theorem \ref{lll} and in particular from \eqref{oooo}  follows that 
\begin{equation}
H^{i}(i_{v*}i_{v}^{*}\mathcal{L}^{*}_{2,max}(C(L)),d_{max,*})\cong H^{i}_{2,max}(reg(L), g_{L})\ \text{for}\ \text{every}\ i=0,...,dimL+1.
\label{atomi}
\end{equation}
But from  theorem \ref{edo} we know that 
\begin{equation}
H^{i}(\mathcal{L}^{*}_{2,max}(C(L)),d_{max,*})\cong H^{i}_{2,max}(reg(L), g_{L})\ for\ i\leq [[\frac{dimL}{2}+\frac{1}{2c}]].
\label{verde}
\end{equation} 
 So  for $i\leq [[\frac{dimL}{2}+\frac{1}{2c}]]$ $$(\mathbb{H}^{i}(i_{v*}i_{v}^{*}\mathcal{L}^{*}_{2,max},d_{max,*}))_{v}\cong (\mathbb{H}^{i}(\mathcal{L}^{*}_{2,max},d_{max,*}))_{v}$$ and therefore we can conclude that for the same $i$ the complexes $(\mathcal{L}^{*}_{2,max},d_{max,*})$\\ and $(i_{v*}i_{v}^{*}\mathcal{L}^{*}_{2,max},d_{max,*})$ are quasi-isomorphic.\\ Now let $j$ be the morphism between $(\mathcal{L}^{*}_{2,max},d_{max,*})$ and $(i_{v*}i_{v}^{*}\mathcal{L}^{*}_{2,max},d_{max,*})$ induced from $i_{v}:C(L)-\{v\}\rightarrow C(L)$. It is immediate to note that for each open subset $U\subset C(L)$ $j_{U}$ is just the inclusion of $\mathcal{L}^{*}_{2,max}(U)$ in $i_{v*}i_{v}^{*} \mathcal{L}^{*}_{2,max}(U)$. Therefore if we call  $j^{*}$  the morphism induced from $j$ between the cohomology sheaves $H^{i}(\mathcal{L}^{*}_{2,max},d_{max,*})$ and $H^{i}(i_{v*}i_{v}^{*} \mathcal{L}^{*}_{2,max},d_{max,*})$ it is immediate to note that $j^{*}$ induces the isomorphism \eqref{frutta}. Finally if we call $\phi$ and $\psi$ respectively the isomorphisms \eqref{atomi} and \eqref{verde} we have that for $i\leq [[\frac{dimL}{2}+\frac{1}{2c}]]$ $$\phi\circ j^{*}=\psi.$$ Therefore we can conclude that  
 $$j:(\mathcal{L}^{*}_{2,max},d_{max,*})\rightarrow (i_{v*}i_{v}^{*}\mathcal{L}^{*}_{2,max},d_{max,*})$$  is a quasi-isomorphism for  $i\leq [[\frac{dimL}{2}+\frac{1}{2c}]]$.
\end{proof}

\begin{cor} Let $(M,h)$ be an oriented riemannian manifold, let $L$ be a compact smoothly stratified pseudomanifold with a Thom-Mather stratification and let $g_{L}$ be a riemannian metric  on $reg(L)$. Consider now $M\times C(L)$ and on $reg(M\times C(L))$ consider the metric $h+dr\otimes dr+r^{2c}g_{L}$. Let $i_{M}:M\times C(L)-(M\times \{v\})\rightarrow M\times C(L)$ the canonical inclusion where $v$ is the vertex of the cone. Finally consider over $M\times C(L)$ the complex of sheaves $(\mathcal{L}^{*}_{2,max}, d_{max,*})$.
 Then the canonical inclusion $$i_{M}:M\times C(L)-(M\times\{v\})\longrightarrow M\times C(L)$$ induces a quasi-isomorphism between   the complexes $$(\mathcal{L}^{*}_{2,max},d_{max,*})\ and\ (i_{M*}i_{M}^{*}\mathcal{L}^{*}_{2,max},d_{max,*})$$  for  $i\leq [[\frac{dimL}{2}+\frac{1}{2c}]]$.
\label{pietro}
\end{cor}

\begin{proof}
The proof is completely analogous to the proof of proposition \ref{bea}.  For every $i=0,...,dimM+dimL+1$ the sheaf $i_{M*}i_{M}^{*}\mathcal{L}^{i}_{2,max}$ is isomorphic to the following sheaf; let $U\subset M\times C(L)$ an open subset then:

 $$i_{M*}i_{M}^{*}\mathcal{L}^{i}_{2,max}(U)\cong\{\omega\in L^{2}_{Loc}\Omega^{i}(reg(U),h+dr\otimes dr+r^{2c}g_{L}|_{reg(U)}): \forall\ p\in U-(U\cap (M\times \{v\})$$    $$\exists\ V\ open\ neighbourhood\ of\ p\ in\ U\ such\ that\ \omega|_{reg(V)}\in \mathcal{D}(reg(V), d_{max,i})\}.$$ From this it follows that for every $x\in M\times C(L)-(M\times \{v\})$ $$(\mathbb{H}^{i}(\mathcal{L}^{*}_{2,max},d_{max,*}))_{x}\cong (\mathbb{H}^{i}(i_{M*}i_{M}^{*}\mathcal{L}^{*}_{2,max},d_{max,*}))_{x}.$$ Now let $p=(m,v)\in M\times \{v\}$. By theorem \ref{lll}, remark \ref{sdds} and proposition \ref{ioi} we know that 
 \begin{equation}
(\mathbb{H}^{i}(\mathcal{L}^{*}_{2,max},d_{max,*}))_{p}\cong H^{i}(\mathcal{L}^{*}_{2,max}(U\times C(L)),d_{max,*})\cong H^{i}_{2,max}(reg(L),g_{L})
\label{klkl}
\end{equation}
 for $i\leq [[\frac{dimL}{2}+\frac{1}{2c}]]$ where $U$ is an open neighborhood of $m$ in $M$  diffeomorphic to an open ball in $\mathbb{R}^{s}$ where $s=dimM$. Moreover, like in  the proof  of the previous proposition, it is easy to show that 
\begin{equation}
(\mathbb{H}^{i}(i_{v*}i_{v}^{*}\mathcal{L}^{*}_{2,max},d_{max,*}))_{p}\cong H^{i}(i_{v*}i_{v}^{*}\mathcal{L}^{*}_{2,max}(U\times C(L)),d_{max,*})
\label{wsws}
\end{equation}
 where $U$ is as in \eqref{klkl}.
 Therefore in order to show that $$(\mathbb{H}^{i}(i_{M*}i_{M}^{*}\mathcal{L}^{*}_{2,max},d_{max,*}))_{p}\cong (\mathbb{H}^{i}(\mathcal{L}^{*}_{2,max},d_{max,*}))_{p}$$ for $i\leq [[\frac{dimL}{2}+\frac{1}{2c}]]$ it is sufficient to show that for the same $i$ $$H^{i}(i_{M*}i_{M}^{*}\mathcal{L}^{*}_{2,max}(U\times C(L)),d_{max,*})\cong H^{i}(\mathcal{L}^{*}_{2,max}(U\times C(L)),d_{max,*})$$ where $U$ is as in \eqref{klkl}.  But from the same observations of the  proof of prop. \ref{bea} and prop. \ref{ioi} follows immediately that $$H^{i}(i_{M*}i_{M}^{*}\mathcal{L}^{*}_{2,max}(U\times C(L)),d_{max,*})\cong H^{i}_{2,max}(reg(L),g_{L})\ \text{for}\ \text{each}\ i$$ and that $$H^{i}(\mathcal{L}^{*}_{2,max}(U\times C(L)),d_{max,*})\cong H^{i}_{2,max}(reg(L), g_{L})\ \text{for}\ i\leq [[\frac{dimL}{2}+\frac{1}{2c}]].$$ So  for $i\leq [[\frac{dimL}{2}+\frac{1}{2c}]]$ $$(\mathbb{H}^{i}(i_{M*}i_{M}^{*}\mathcal{L}^{*}_{2,max},d_{max,*}))_{p}\cong (\mathbb{H}^{i}(\mathcal{L}^{*}_{2,max},d_{max,*}))_{p}$$ and therefore we can conclude that for the same $i$ the complexes $(\mathcal{L}^{*}_{2,max},d_{max,*})$\\ and $(i_{M*}i_{M}^{*}\mathcal{L}^{*}_{2,max},d_{max,*})$ are quasi-isomorphic. Now using the same final considerations of the previous proof we get the conclusion.
\end{proof}

Finally we can give the proof of the theorem announced at the beginning of the section:

\begin{proof} (of  theorem \ref{ris}). 
Using  corollary \ref{trucco}  we know that there is a quasi rigid iterated edge metric on reg(X), $g'$, that is  quasi-isometric to $g$. So, without loss of generality, we can suppose that  $g$ is a quasi rigid iterated  edge metric  with weights.
We start by proving the isomorphism \ref{max}. The proof is given by induction on the depth of $X$. If $depth(X)=0$ there is nothing to show because, in this case, $X$ is a closed manifold and therefore the isomorphisms \ref{max} are the well know theorems of Hodge and de Rham. Suppose now that the theorem is true if $depth(X)\leq k-1$ and that $depth(X)=k$. We will show that the theorem is also true in this case.  We begin showing the first isomorphism, $H_{2,max}^{i}(reg(X),g)\cong I^{q_{g}}H^{i}(X,\mathcal{R}_{0})$; to do this we will use theorem \ref{tre}, corollary \ref{mlp} and remark \ref{ooo}. More precisely we will show that the complex $(\mathcal{L}^{i}_{2,max}, d_{max,i})$ satisfies the three axioms of theorem \ref{tre} respect to the perversity $p_{g}$, the stratification $\mathfrak{X}$ and the local system over $reg(X)$ given by $\mathcal{R}\otimes \mathcal{O}$ where $\mathcal{R}$ is $(X-X_{n-1})\times \mathbb{R}$ with $\mathbb{R}$ endowed of the discrete topology  and $\mathcal{O}$ is the orientation sheaf (see example \ref{ytr}). By proposition \ref{cfs} we know that $(\mathcal{L}^{i}_{2,max}, d_{max,i})$ is a complex of fine sheaves. The first two requirements of axiom 1 are clearly satisfied. The third requirement of the same axiom follows by proposition \ref{ioi} wich implies  that for each $x\in reg(X)$ $(\mathbb{H}^{i}(\mathcal{L}^{*}_{2,max},d_{max,*}))_{x}$, that is the stalk at the point $x$ of the $i-$th cohomology sheaf associated to the complex $(\mathcal{L}^{*}_{2,max}, d_{max,*})$, satisfies:
\begin{equation}
(\mathbb{H}^{i}(\mathcal{L}^{*}_{2,max},d_{max,*}))_{x}= \left\{
\begin{array}{ll}
\mathbb{R} & i=0\\
0 & i>0
\end{array}
\right.
\end{equation}
Consider now a stratum $Y\subset X$ and a point $x\in Y$. Let $l=dimY$. If $l=n-1$, that is if the codimension of $Y$ is $1$, then it is clear from proposition \ref{ioi} that for all $x\in Y$ the second axiom of theorem \ref{tre} is satisfied. So we can suppose that $l\leq n-2$.   
 By definition \ref{edge} we know that there exists an open subset  $V\subset Y$ such that $\pi_{Y}^{-1}(V)\cong V\times C(L_{Y})$
and such that $$\phi:(\pi_{Y}^{-1}(V)\cap reg(X), g|_{\pi_{Y}^{-1}(V)\cap reg(X)})\rightarrow (V\times reg(C(L_{Y})), dr^{2}+h_{V}+r^{2c_{Y}}g_{L_{Y}})$$ is a quasi-isometry. Therefore by the invariance of $L^{2}-$cohomology under quasi-isometry we can use $(V\times reg(C(L_{Y})), dr^{2}+h_{V}+r^{2c_{Y}}g_{L_{Y}})$ to calculate the $L^{2}-$cohomology of $\pi_{Y}^{-1}(V)\cap reg(X)$. Choosing $V$ diffeomorphic to $(0,\epsilon)^{l}$ with $\epsilon$ sufficiently small we have that 
\begin{equation}
\label{hhuu}
(V\times reg(C(L_{Y})), dr^{2}+h_{V}+r^{2c_{Y}}g_{L_{Y}})
\end{equation}
 is quasi-isometric to $$((0,\epsilon)^{l}\times reg(C(L_{Y})),ds_{1}^{2}+...+ds_{l}^{2}+dr^{2}+r^{2c_{Y}}g_{Y}).$$ Therefore  from proposition \ref{ioi} and the invariance of $L^{2}-$cohomology under quasi-isometry it follows that 
\begin{equation}
\label{kjl}
H^{i}_{2,max}(V\times reg(C(L_{Y})), dr^{2}+h_{V}+r^{2c_{Y}}g_{L_{Y}})\cong H_{2,max}^{i}(reg(C(L_{Y})),dr^{2}+r^{2c_{Y}}g_{L_{Y}}).
\end{equation}
In this way  we have obtained that 
\begin{equation}
H^{i}_{2,max}(reg(\pi_{Y}^{-1}(V)),g|_{reg(\pi_{Y}^{-1}(V))})\cong H_{2,max}^{i}(reg(C(L_{Y})),dr^{2}+r^{2c_{Y}}g_{L_{Y}}).
\label{jhg}
\end{equation}


As we have already observed in the proof of corollary \ref{pietro} we know that $$(\mathbb{H}^{i}(\mathcal{L}^{*}_{2,max},d_{max,*}))_{x}\cong H^{i}_{2,max}(reg(\pi_{Y}^{-1}(V)),g|_{reg(\pi_{Y}^{-1}(V))})$$ where $V$ is as in \ref{hhuu}.
Therefore from this and \eqref{jhg} we get that 
\begin{equation}
(\mathbb{H}^{i}(\mathcal{L}^{*}_{2,max},d_{max,*}))_{x}\cong H^{i}_{2,max}(reg(C(L_{Y})),dr\otimes dr+r^{2c_{Y}}g_{L_{Y}})
\end{equation}
Now, using the inductive hypothesis we know that this theorem is true for $(L_{Y},g_{L_{Y}})$ that is $H^{i}_{2,max}(reg(L_{Y}), g_{L_{Y}})\cong I^{q_{g_{L_{Y}}}}H^{i}(L_{Y}, \mathcal{R}_{0})$ where $q_{g_{L_{Y}}}=t-p_{g_{L_{Y}}}$ and $p_{g_{L_{Y}}}$ is the general perversity associated to $g_{L_{Y}}$ on $L_{Y}$. This implies that $dimH^{i}_{2,max}(reg(L_{Y}),g_{L_{Y}})<\infty$ for each $i=0,...,dimL_{Y}$. From this it follows that at least one of the three hypotheses of  corollary \ref{edo} is always satisfied.  So we can use the same  corollary  to get:
\begin{equation}
H^{i}_{2,max}(reg(C(L_{Y})),g_{c}) = \left\{  
\begin{array}{ll}
H_{2,max}^{i}(reg(L_{Y}),g_{L_{Y}}) & i<\frac{dimL_{Y}}{2}+\frac{1}{2c_{Y}}\\
0& i\geq\frac{dimL_{Y}}{2}+\frac{1}{2c_{Y}}
\end{array}
\right.
\label{jlk}
\end{equation}

In this way we can conclude that for each $x\in Y$ $$(\mathbb{H}^{i}(\mathcal{L}_{2,max}^{*},d_{max,*}))_{x}=0\ for\ i>p_{g}(Y)$$ and therefore the complex $(\mathcal{L}^{*}_{2,max}, d_{max,*})$ satisfies the second axiom of theorem \ref{tre}.\\To conclude the first part of the proof we have to show that given any stratum $Z\subset X_{n-k}-X_{n-k-1}$ and any point $x\in Z$  the attaching map, that is the morphism given by the composition of $$\mathcal{L}^{*}_{2,max}|_{U_{k+1}}\rightarrow i_{k*}\mathcal{L}^{*}_{2,max}|_{U_{k}}\rightarrow Ri_{k*}\mathcal{L}_{2,max}^{*}|_{U_{k}}$$ where the first morphism is induced by the inclusion $i_{k}:U_{k}\rightarrow U_{k+1}$, is a quasi-isomorphism at $x$ up to $p_{g}(Z).$ By the fact that  $(\mathcal{L}^{*}_{2,max}, d_{max,*})$ is a complex of fine sheaves it follows that $i_{k*}\mathcal{L}^{*}_{2,max}|_{U_{k}}\rightarrow Ri_{k*}\mathcal{L}_{2,max}^{*}|_{U_{k}}$ is a quasi-isomorphism (for example see \cite{B} pag. 32 or \cite{BL} pag. 222). Therefore, to conclude, we have only to show that the morphism $\mathcal{L}^{*}_{2,max}|_{U_{k+1}}\rightarrow i_{k*}\mathcal{L}^{*}_{2,max}|_{U_{k}}$ is a quasi-isomorphism at $x$ up to $p_{g}(Z)$, that is, for each $x\in Z$ it induces an isomorphism 
\begin{equation}
(\mathbb{H}^{i}(\mathcal{L}^{*}_{2,max}|_{U_{k+1}}, d_{max,*}))_{x}\cong (\mathbb{H}^{i}(i_{k*}\mathcal{L}^{*}_{2,max}|_{U_{k}}, d_{max,*}))_{x}\ for\ i\leq p_{g}(Z).
\label{fritto}
\end{equation}

Now, like in the previous case to prove the validity of the second axiom, to show that for each $x\in Z$ $$(\mathbb{H}^{i}(\mathcal{L}^{*}_{2,max}|_{U_{k+1}}, d_{max,*}))_{x}\cong (\mathbb{H}^{i}(i_{k*}\mathcal{L}^{*}_{2,max}|_{U_{k}}, d_{max,*}))_{x}\ for\ i\leq p_{g}(Z)$$ it is sufficient to show that there exists an open neighbourhood  $U$ of $x\in Z$ such that $\pi_{Z}^{-1}(U)\cong U\times C(L_{Z})$ and such that   $$H^{i}(\mathcal{L}^{*}_{2,max}|_{U_{k+1}}(\pi_{Z}^{-1}(U)), d_{max,*})\cong H^{i}(i_{k*}\mathcal{L}^{*}_{2,max}|_{U_{k}}(\pi_{Z}^{-1}(U)), d_{max,*})\ for\ i\leq p_{g}(Z)$$ where the isomorphism is induced by  the inclusion $i_{k}:U_{k}\rightarrow U_{k+1}$.
Finally   this last statement follows from corollary \ref{pietro}.  
 So given a stratum $Z\subset X_{n-k}-X_{n-k-1}$ and a point $x\in Z$ 
  we can conclude that for  $i\leq p_{g}(Z)$ the natural maps induced by the inclusion of $U_{k}$ in $U_{k+1}$ induces a quasi isomorphism between $$\mathcal{L}^{*}_{2,max}|_{U_{k+1}}\rightarrow i_{k*}\mathcal{L}^{*}_{2,max}|_{U_{k}}.$$ So also the third axiom of theorem \ref{tre} is satisfied.\\ Therefore for all $i=0,...,n$ $H^{i}(\mathcal{L}_{2,max}(reg(X)),d_{max,*})\cong I^{q_{g}}H^{i}(X,\mathcal{R}_{0})$. Finally by the compactness of $X$, see the third point of proposition \ref{ddd}, we get, for each $i=0,...,n$, the desired isomorphisms: $$H_{2,max}^{i}(reg(X),g)\cong I^{q_{g}}H^{i}(X,\mathcal{R}_{0}).$$

From the isomorphism $H_{2,max}^{i}(reg(X),g)\cong I^{q_{g}}H^{i}(X, \mathcal{R}_{0})$ it follows that $H^{i}_{2,max}(reg(X),g)$ is finite dimensional and then the isomorphism $\mathcal{H}_{abs}^{i}(reg(X))\cong H_{2,max}^{i}(reg(X),g)$ is an immediate consequence of proposition \ref{qwe} and formula \ref{kd}. The first part of the proof is completed.\vspace{1 cm}
To prove the second part of the theorem it is sufficient observe that the finite dimension of $H^{i}_{2,max}(reg(X),g)$ for all $i=0,...,n$  implies that the complex $(L^{2}\Omega^{*}(reg(X),g),d_{max,*})$ is a Fredholm complex. Now, using the isomorphism induced by the Hodge star operator $*$ between the Hilbert complexes $(L^{2}\Omega^{*}(reg(X),g), d_{min,*})$ and the adjoint complex of $(L^{2}\Omega^{*}(reg(X),g), d_{max,*})$ and proposition \ref{fred}, it follows that $$H^{i}_{2,max}(reg(X),g)\cong H^{n-i}_{2,min}(reg(X), g).$$ Finally, using  Poincaré duality for intersection homology, that is theorem \ref{mzb}, we get the isomorphism $$H_{2,min}^{i}(reg(X),g)\cong I^{p_{g}}H^{i}(X, \mathcal{R}_{0}).$$ Now, like in the previous case, we know that $H_{2,min}^{i}(reg(X),g)$ is finite dimensional and then the isomorphism $\mathcal{H}_{rel}^{i}(reg(X))\cong H_{2,min}^{i}(reg(X),g)$ is an immediate consequences of proposition \ref{qwe} and formula \ref{kd}. 
\end{proof}

\begin{proof} (of theorem \ref{new}). Suppose that $p$ is a general perversity in the sense of Friedman on $X$ such that $p\geq \overline{m}$ and $p(Y)=0$ for each one codimensional stratum $Y$ of $X$. We recall that $\overline{m}$ is defined in the following way: if $Y\subset X$ is a stratum of $X$ and if $L_{Y}$ is the link relative to $Y$ with $l_{Y}=dimL_{Y}$ then 
$$\overline{m}(Y)= \left\{
\begin{array}{ll}
\frac{l_{Y}}{2} & l_{Y}\ even\\
\frac{l_{Y}-1}{2} & l_{Y}\ odd
\end{array}
\right.
$$
Therefore  it follows that for each stratum $Y$ there is a non negative integer $n_{Y}$ such that 
$$p(Y)= \left\{
\begin{array}{lll}
0 & l_{Y}=0\\
\frac{l_{Y}}{2}+n_{Y} & l_{Y}\ even,\ l_{Y}\neq 0\\
\frac{l_{Y}-1}{2}+n_{Y} & l_{Y}\ odd
\end{array}
\right.
$$
Now we can choose some non negative real numbers $\{c_{Y}\}_{Y\in \mathfrak{X}}$ such that $n_{Y}=[[\frac{1}{2c_{Y}}]]$ if $l_{Y}$ is even and $n_{Y}=[[\frac{1}{2}+\frac{1}{2c_{Y}}]]$ if $l_{Y}$ is odd. By proposition \ref{top} we know that there is a quasi rigid iterated edge metric $g$ on $reg(X)$ having the numbers $\{c_{Y}\}_{Y\in \mathfrak{X}}$ like weights. In this way $p=p_{g}$, the general perversity associated to $g$, and therefore by theorem \ref{ris} we can get the isomorphism \eqref{minn} .\\Conversely if $p$ satisfies $p\leq \underline{m}$ and $p(Y)=-1$ for each one codimensional stratum $Y$ of $X$, then $q:=t-p$, where $t$ is top perversity, satisfies $q\geq \overline{m}$ and $q(Y)=0$ for each one codimensional stratum $Y$ of $X$. Therefore by the previous point there exists a quasi edge metric with weights $h$ on $reg(X)$ such that $p_{h}=q$. Finally using again theorem \ref{ris} we can get the isomorphism \eqref{maxx}.
\end{proof}

In the same hypothesis of the  theorem \ref{ris} we have the following  corollaries:

\begin{cor} For each $i=0,...,n$ on $L^{2}\Omega^{i}(reg(X),g)$ we have the following decompositions:
\begin{equation}
L^{2}\Omega^{i}(reg(X),g)=\mathcal{H}^{i}_{abs}\oplus ran(d_{max,i-1})\oplus ran(\delta_{min,i})
\label{ass}
\end{equation}
\begin{equation}
L^{2}\Omega^{i}(reg(X),g)=\mathcal{H}^{i}_{rel}\oplus ran(d_{min,i-1})\oplus ran(\delta_{max,i})
\label{asso}
\end{equation}and
\begin{equation}
L^{2}\Omega^{i}(reg(X),g)=\mathcal{H}^{i}_{max}\oplus ran(d_{min,i-1})\oplus ran(\delta_{min,i})
\label{tyt}
\end{equation}
\label{zmmz}
\end{cor}

\begin{proof} By theorem \ref{ris} we know that $H^{i}_{2,max}(reg(X),g)$ and $H^{i}_{2,min}(reg(X),g)$ are finite dimensional. Therefore by prop. \ref{qwe}, the fact that $(L^{2}\Omega^{*}(M,g),\delta_{min,*})$ is the dual complex of $(L^{2}\Omega^{*}(M,g),d_{max,*})$, $(L^{2}\Omega^{*}(M,g),\delta_{max,*})$ is the dual complex of $(L^{2}\Omega^{*}(M,g),d_{min,*})$ and proposition \ref{fred} it follows that, for each $i$, $ran(d_{max,i}),ran(d_{min,i}),ran(\delta_{max,i})$ and $ran(\delta_{min,i})$ are closed. Now applying \eqref{kd} we can get \eqref{ass} and \eqref{asso} and applying \eqref{oiko} we can get \eqref{tyt}. 
\end{proof}

\begin{cor} $$d_{max}+\delta_{min},\ d_{min}+\delta_{max}: L^{2}\Omega^{*}(reg(X),g)\rightarrow L^{2}\Omega^{*}(reg(X),g)$$  and for each $i$ $$\Delta_{abs,i},\  \Delta_{rel,i}:L^{2}\Omega^{i}(reg(X),g)\rightarrow L^{2}\Omega^{i}(reg(X),g)$$ are Fredholm operators. 
Moreover also  $$d_{max}+\delta_{min},\ d_{min}+\delta_{max}: L^{2}\Omega^{even}(reg(X),g)\rightarrow L^{2}\Omega^{odd}(reg(X),g)$$ are Fredholm operators and their indexes satisfy: $$ind(d_{max}+\delta_{min})=\sum_{i=0}^n(I^{q_{g}}b_{2i}(X)-I^{p_{g}}b_{2i+1}(X))$$  $$ind(d_{min}+\delta_{max})=\sum_{i=0}^n(I^{p_{g}}b_{2i}(X)-I^{q_{g}}b_{2i+1}(X))$$ where $I^{p_{g}}b_{2i}(X)=dim(I^{p_{g}}H^i(X,\mathcal{R}))$ and analogously $I^{q_{g}}b_{2i}(X)=dim(I^{q_{g}}H^i(X,\mathcal{R}))$.\\
Finally $$\Delta_{max,i}:L^{2}\Omega^{i}(reg(X),g)\rightarrow L^{2}\Omega^{i}(reg(X),g)$$ has closed range and its orthogonal complement is finite dimensional while  $$\Delta_{min,i}:L^{2}\Omega^{i}(reg(X),g)\rightarrow L^{2}\Omega^{i}(reg(X),g)$$ has closed range and finite dimensional nullspace; in other words $\Delta_{max,i}$ is essentially surjective and $\Delta_{min,i}$ is essentially injective.
\end{cor}

\begin{proof} The first three assertions follow immediately from theorem \ref{ris}. For the last two we know that $ran(\Delta_{abs,i})\subset ran(\Delta_{max}).$ This implies that there exists a surjective map from $$\frac{L^{2}\Omega^{i}(M,g)}{ran(\Delta_{abs,i})}\longrightarrow \frac{L^{2}\Omega^{i}(M,g)}{ran(\Delta_{max,i})}.$$ But we know that $\Delta_{abs}$ is Fredholm; this implies that the term on the left in the above equality is finite dimensional and therefore also the term on the right is finite dimensional. So $\Delta_{max,i}$ from its natural domain endowed with the graph norm to $L^{2}\Omega^{i}(M,g)$ is a continuous operator with finite dimensional cokernel and this implies the  statement of the corollary about $\Delta_{max,i}$. For $\Delta_{min,i}$ we know, see prop. \ref{olpo}, that $Ker(\Delta_{min,i})=Ker(d_{min,i})\cap Ker(\delta_{min,i-1})$ and therefore by theorem \ref{ris} it follows that $Ker(\Delta_{min,i})$ is finite dimensional. Using again proposition \ref{olpo} we know that $(\Delta_{max,i})^*=\Delta_{min,i}$ and therefore by the fact that $\Delta_{max,i}$ has closed range it follows that also $\Delta_{min,i}$ has closed range.
\end{proof}

Finally the remaining corollaries follow immediately from theorem \ref{ris} and from the definition of  intersection cohomology with general perversity.

\begin{cor} Consider the following complex $(C_{0}^{\infty}\Omega^{i}(reg(X)),d_{i})$. Then a necessary condition to have the minimal exstension equal to the maximal one is that the perversities $p_{g}$ and $q_{g}$ gives isomorphic intersection cohomology groups. 
\end{cor}

\begin{cor} If every weight is greater or equal than $1$, that is for every stratum $Y$ $c_{Y}\geq 1$, then, for all $i$, we obtain the following isomorphisms:
\begin{equation}
\mathcal{H}_{abs}^{i}(reg(X),g)\cong H_{2,max}^{i}(reg(X),g)\cong I^{\underline{m}}H^{i}(X,\mathcal{R}_{0})
\label{mal}
\end{equation}
\begin{equation}
\mathcal{H}_{rel}^{i}(reg(X),g)\cong H_{2,min}^{i}(reg(X),g)\cong I^{\overline{m}}H^{i}(X, \mathcal{R}_{0})
\end{equation} where $\underline{m}$ is the lower middle perversity and $\overline{m}$ is the upper middle perversity.
\label{mil}
\end{cor}

\begin{cor} Suppose  that  the general perversity associated to the quasi edge metric with weights $g$ satisfies  $p_{g}(Z)\geq cod(Z)-1$ for each singular stratum $Z$. Then, for all $i$, we have the following isomorphisms:
\begin{equation}
\mathcal{H}_{abs}^{i}(reg(X),g)\cong H_{2,max}^{i}(reg(X),g)\cong H^{i}(X-X_{n-1},\mathbb{R})
\end{equation}
\begin{equation}
\mathcal{H}_{rel}^{i}(reg(X),g)\cong H_{2,min}^{i}(reg(X),g)\cong H^{i}(X, \mathcal{R}_{0}).
\end{equation} 
\end{cor}

\begin{cor} If $p_{g}$ is classical perversity in the sense of Goresky-MacPherson and $X_{n-1}=X_{n-2}$ then, for all $i$, we have the following isomorphisms:
\begin{equation}
\mathcal{H}_{abs}^{i}(reg(X),g)\cong H_{2,max}^{i}(reg(X),g)\cong I^{q_{g}}H^{i}(X,\mathbb{R})
\label{mao}
\end{equation}
\begin{equation}
\mathcal{H}_{rel}^{i}(reg(X),g)\cong H_{2,min}^{i}(reg(X),g)\cong I^{p_{g}}H^{i}(X, \mathbb{R})
\end{equation}
\label{z}
\end{cor}

\begin{cor} Let $g, h$ be two quasi edge metrics with weights on $reg(X)$ such that $p_{g}=p_{h}$. Then for all $i$
\begin{equation} 
\mathcal{H}^{i}_{abs}(reg(X),g)\cong H^{i}_{2,max}(reg(X),g)\cong H^{i}_{2,max}(reg(X),h)\cong \mathcal{H}^{i}_{abs}(reg(X),h)
\end{equation} 
  and
\begin{equation} 
\mathcal{H}^{i}_{rel}(reg(X),g)\cong H^{i}_{2,min}(reg(X),g)\cong H^{i}_{2,min}(reg(X),h)\cong \mathcal{H}^{i}_{rel}(reg(X),h)
\end{equation}
In particular a necessary condition for two quasi edge metrics with weights  are quasi-isometric is that they induce perversities with isomorphic intersection cohomology groups. 
\end{cor}

\begin{cor} Let $X'$ be another compact and oriented smoothly stratified pseudomanifold with a Thom-Mather stratification and $h$ a quasi edge metric with weights on $reg(X')$. Let $f:X\rightarrow X'$ a stratum preserving homotopy equivalence, see \cite{KW} pag 62 for the definition. Suppose that both $p_{g}$ and $p_{h}$ depend only on the codimension of the strata and that $p_{g}=p_{h}$. Then for all $i$

\begin{equation} 
\mathcal{H}^{i}_{abs}(reg(x),g)\cong H^{i}_{2,max}(reg(X),g)\cong H^{i}_{2,max}(reg(X'),h)\cong \mathcal{H}^{i}_{abs}(reg(X'),h)
\end{equation} 
  and
\begin{equation} 
\mathcal{H}^{i}_{rel}(reg(x),g)\cong H^{i}_{2,min}(reg(X),g)\cong H^{i}_{2,min}(reg(X'),h)\cong \mathcal{H}^{i}_{rel}(reg(X'),h)
\end{equation}
\end{cor}

\vspace{1 cm}

\textbf{Acknowledgments.} We wish to thank Paolo Piazza for  suggesting this problem, for many interesting discussions and for his help. We wish also to thank Rafe Mazzeo, Eugenie Hunsicker, Greg Friedman, Pierre Albin and Erich Leicthnam for helpful
conversations and emails.


Dipartimento di Matematica, Sapienza Università di Roma

\texttt{E-mail address}: bei@math.uniroma1.it

\end{document}